\documentclass{amsart}
\input cyracc.def
\font\tencyr=wncyr8
\def\cyr{\tencyr\cyracc}
\usepackage{graphicx, color,amsmath,amstext,amssymb}
\newtheorem{thm}{Theorem}[section]
\newtheorem{cor}[thm]{Corollary}

\newtheorem{lemma}[thm]{Lemma}
\newtheorem{prop}[thm]{Proposition}
\newtheorem{conj}[thm]{Conjecture}
\theoremstyle{definition}
\newtheorem{defn}[thm]{Definition}
\theoremstyle{remark}

\numberwithin{equation}{section}

\newcommand{\M}{\mathcal M}

\newcommand{\Fin}{\textbf{Fin}}

\newcommand{\Set}[2]{\ensuremath{ \{ #1 : #2 \} }}

\definecolor{myPurple}{rgb}{.5,0,.5}

\def\res{\!\!\upharpoonright\!\!}

\newcommand{\la}{\langle}
\newcommand{\ra}{\rangle}
\newcommand{\dom}[1]{\text{dom}(#1)}
\newcommand{\rg}[1]{\text{rg}(#1)}

\def\diverges{\!\uparrow}
\def\converges{\!\downarrow}

\newcommand{\gtilde}{\tilde{g}}
\newcommand{\Gtilde}{\tilde{G}}

\renewcommand{\tt}{{\texttt{tt}}}
\newcommand{\bT}{{\texttt{bT}}}
\newcommand{\wtt}{{\texttt{wtt}}}
\newcommand{\btt}{{\texttt{btt}}}
\newcommand{\kbtt}{{k\text{-}\btt}}
\newcommand{\onebtt}{{1\text{-}\btt}}
\newcommand{\twobtt}{{2\text{-}\btt}}
\newcommand{\normbtt}[1]{{#1\text{-}\btt}}
\def\phi{\varphi}

\newenvironment{pf}{\begin{trivlist}\item[\hskip\labelsep
{\it Proof.}]}{\end{trivlist}}

\newenvironment{pftitle}[1]{\begin{trivlist}\item[\hskip\labelsep
{\it #1.}]}{\end{trivlist}}

\begin{document}

\title[Distance Function]{The distance function on a computable graph}%
\author{Wesley Calvert}%
\address{Department of Mathematics, Mail Code 4408\\ 1245 Lincoln
  Drive\\ Southern Illinois University\\ Carbondale, Illinois 62901, U.S.A.}%
\email{wcalvert@siu.edu}%

\author{Russell Miller}%
\address{Department of Mathematics, Queens College -- CUNY, 65-30 Kissena Blvd., Flushing, NY 11367, U.S.A.; and Ph.D.\ Programs in Mathematics and Computer Science, CUNY Graduate Center, 365 Fifth Avenue, New York, NY  10016, U.S.A.}%
\email{Russell.Miller@qc.cuny.edu}%

\author{Jennifer Chubb Reimann}%
\address{Department of Mathematics\\University of San Francisco\\2130 Fulton Street\\San Francisco, California 94117, U.S.A.}%
\email{jcchubb@usfca.edu}%

\thanks{The authors wish to acknowledge useful conversations with
  Denis Hirschfeldt.  The work of the first author was partially
  supported by a Fulbright-Nehru Senior Research Scholarship, the
  George Washington University Research Enhancement Fund, and NSF
  grant DMS--1101123.  That of the second author was partially supported
  by grant DMS--1001306 from the National Science Foundation,
  by grant 13397 from the Templeton Foundation, and by grants numbered
  62632-00 40, 63286-00 41, and 64229-00 42 from
  The City University of New York PSC-CUNY Research Award Program.  The third author had partial support from The George Washington University Research Enhancement Fund and from NSF grant DMS--0904101. }
\begin{abstract}
We apply the techniques of computable model theory to the distance
function of a graph.  This task leads us to adapt the definitions of several
truth-table reducibilities so that they apply to functions as well as to sets,
and we prove assorted theorems about the new reducibilities and about functions
which have nonincreasing computable approximations.  Finally, we show
that the spectrum of the distance function can consist of an arbitrary single
\btt-degree which is approximable from above, or of all such \btt-degrees at once,
or of the \bT-degrees of exactly those functions approximable from above
in at most $n$ steps.
\end{abstract}
\maketitle
\section{Introduction}

Every connected graph has a distance function, giving the length
of the shortest path between any pair of nodes in the graph.
Graphs appear in a wide variety of mathematical applications,
and the computation of the distance function is usually
crucial to these applications.  Examples range from web
search engine algorithms,
to Erd\"os numbers and parlor games (``Six Degrees of Kevin Bacon''),
to purely mathematical questions.

Therefore, the question of the difficulty of computing
the distance function is of natural interest to mathematicians
in many areas.  This article is dedicated to exactly that enterprise,
on infinite graphs.
Assuming that the graph in question is symmetric,
irreflexive, and computable -- that is,
that one can list out all its nodes and decide effectively which pairs
of nodes have an edge between them --
we investigate the Turing degree and other measures
of the difficulty of computing the distance function.

We began this study by considering the spectrum of the distance
function -- a standard concept in computable model theory,
giving the set of the Turing degrees of distance functions on
all computable graphs isomorphic to the given graph.
This notion is usually used for relations on a computable structure,
rather than for functions, but it is certainly the natural first question
one should ask.  As our studies continued, however, they led us
to consider finer reducibilities than ordinary Turing reducibility,
and since we were studying a function instead of a relation,
we often had to adapt these reducibilities to functions.
The resulting concepts are likely to be of interest to pure computability
theorists, as well as to those dealing with applications, and,
writing the paper in logical rather than chronological order,
we spend the first sections defining and examining these
reducibilities on functions.  Only in the final sections do we address
the original questions about the distance function on a computable graph.
Therefore, right here we will offer some further intuition
about the distance function, to help the reader understand
why the material in the first few sections is relevant.

For a computable connected graph $G$, the natural first
step for approximating the distance $d(x,y)$ between two nodes
$x,y\in G$ is to find some path between them.  By connectedness,
a systematic search is guaranteed to produce such a path
sooner or later, and its length is our first approximation
to the distance from $x$ to $y$.  The next natural step is to
search for a shorter path, and then a shorter one than that,
and so on.  Of course, these path lengths are all just approximations
to the actual distance from $x$ to $y$.  One of the approximations
will be correct, and once we find it, its path length will never
be superseded by any other approximation.  That is, our
(computable) approximations will \emph{converge} to the correct
answer, and so the distance function is always $\emptyset'$-computable,
by the Limit Lemma (see \cite[Lemma III.3.3]{S87}).
However, unless $d(x,y)\leq 2$, we will never be sure that our approximation
is correct, since a shorter path could always appear.

By definition
$d(x,x)=0$, and $d(x,y)=1$ iff $x$ and $y$ are adjacent, so it is computable
whether either of these conditions holds.  It is not in general computable
whether $d(x,y)=2$, but it is $\Sigma^0_1$, since we need only
find a single node adjacent to both $x$ and $y$.  For each $n>2$,
however, the condition $d(x,y)=n$ is given by the conjunction
of a universal formula and an existential formula, hence defines a
difference of computably enumerable sets, and in general cannot
be expressed in any simpler way than that.  Indeed, the distance function
often fails to be computable, and likewise the set of pairs $(x,y)$
with $d(x,y)=n$ often fails to be computably enumerable when $n>2$.
In what follows, however, we will show that the distance function
always has computably enumerable Turing degree -- which in turn
will start to suggest why Turing equivalence is not the most
useful measure of complexity for these purposes.  Other questions
immediately arise as well.  For instance, must there exist a computable
graph $H$ isomorphic to $G$ such that $H$ has computable distance function?
Or at least, must there exist a computable $H\cong G$ such that
we can approximate the distance function on $H$ with no more than
one (or $n$) wrong answer(s)?

The approximation algorithm described above is not of arbitrary difficulty,
in the pantheon of computable approximations to functions.  Our approximations
to $d(x,y)$ are always at least $d(x,y)$, and decrease until
(at some unknown stage) they equal $d(x,y)$.
Hence $d(x,y)$ is \emph{approximable from above}.  This notion
already exists in the literature; the most commonly seen function of this type
is probably Kolmogorov complexity, which on input $n$ gives the shortest length
of a program outputting $n$.  Up to the present, approximability from above
has not been so common in computable model theory; the more common notion
there is \emph{approximability from below}, which arises (for instance)
when one tries to find the number of predecessors of a given element
in a computable linear order of order type $\omega$.
In Section \ref{sec:AFA} we give full definitions of the class
of functions approximable from above,
which can be classified in much the same way as
the Ershov hierarchy, and compare it with the class of functions
approximable from below.  These two classes turn out to be more
different than one would expect!  First, though, in Section
\ref{sec:reducibilities}, we present the exact definitions
of the reducibilities we will use on our functions, so that we may refer
to these reducibilities in Section \ref{sec:AFA}.

Very little background in computable model theory is actually
required in order to read this article, since the distance function
turned out to demand a somewhat different approach than is typical in that field of study.
A background in general computability theory will be useful, however,
particularly in regard to several of the so-called \emph{truth-table reducibilities},
and for this we suggest \cite{Odi}, to which we will refer frequently
in Section \ref{sec:reducibilities}.  We try to maintain notation
from \cite{Odi} as we adapt the definitions of the truth-table
reducibilities to deal with functions.  We give the requisite definitions
about graphs in Section \ref{sec:graphs}, where they are first needed.

\section{Reducibilities on Functions}
\label{sec:reducibilities}

When discussing Turing computability relative to an oracle,
mathematicians have traditionally
taken the oracle to be a subset of $\omega$.  To compute relative
to a total or partial function from $\omega^n$ into $\omega$,
they simply substitute the graph of the function for the function itself,
then apply a coding of $\omega^{n+1}$ into $\omega$.
This metonymy works admirably as far as ordinary
Turing reducibility is concerned, and any alternative definition
of Turing reducibility for functions should be equivalent to this one.
However, \emph{bounded} Turing reducibility (in which the use of the oracle is computably bounded; see Definition \ref{defn:tt}) among functions requires a
new definition, and so we offer here an informal version of our notion
of a function oracle.

First, to motivate this notion, consider
bounded Turing ($\bT$) reducibility with a function oracle.  One would certainly
assume that a total function $f$ should be $\bT$-reducible
to itself.  However, if one wishes to compute $f(x)$
for arbitrary $x$, using the graph of $f$ as an oracle, and if $f$
is not computably bounded, then there is no obvious way
to compute in advance an upper bound on the codes $\la m,n\ra$
of pairs for which one will have to ask the oracle about membership
of that pair in the graph.  (This question is addressed more
rigorously in Proposition \ref{prop:compbdd}
and Lemma \ref{lemma:Fokina} below.)
So, when we move to reducibilities finer than $\leq_T$,
there is a clear need for a notion of Turing machine having
a function as the oracle.

For simplicity, we conceive of a Turing machine
with a function oracle $F$ as having three tapes:  a two-way scratch tape
(on which the output of the computation finally appears, should
the computation halt), a one-way \emph{question tape}, and a one-way
\emph{answer tape}.  When the machine wishes to ask its function
oracle for the value $F(x)$ for some specific $x$, it must write
a sequence of exactly $(x+1)$ $1$'s on the question tape,
and must set every cell of the answer tape to be blank.
Then it executes the \emph{oracle instruction}.  In this model
of computation, the oracle instruction is no longer a forking instruction.
Rather, with the tapes in this state, the oracle instruction
causes exactly $(1+F(x))$ $1$'s to appear on the answer tape
at the next step, and the machine simply proceeds to the next
line in its program (which most likely will start counting
the number of $1$'s on the answer tape, in order to use the
information provided by the oracle).  Notice that it is
perfectly acceptable to say that a set is computable
from a function oracle (using the above notion),
or that a function is computable from a set oracle (using the usual
notion of oracle Turing computation).  If one wishes to speak only of functions,
there is no harm in replacing a set by its characteristic function.
We leave it to the compulsive reader to formulate the precise definition
of a Turing machine with function oracle, by analogy to the
standard definition for set oracles.  A more immediate
(and equivalent) definition uses a different approach.
\begin{defn}
\label{defn:fctoracle}
The class of \emph{partial functions on $\omega$ computable
with a function oracle $F$}, where $F:\omega\to\omega$ is total,
is the smallest class of partial functions closed under the axiom
schemes I-VI from \cite[\S~I.2]{S87}
and containing $F$.
\end{defn}
Of course, this is exactly the usual definition of the partial $F$-recursive
functions, long known to be equivalent to the definition of functions
computable by a Turing machine with the graph of $F$ as its oracle.
As far as Turing reducibility is concerned, nothing has changed.
Only stronger reducibilities need to be considered.  We start
by converting the standard definitions (for sets) of $\leq_m$
and $\leq_1$ to definitions for functions.

\begin{defn}
\label{defn:mreducible}
Let $\phi$ and $\psi$ be partial functions from $\omega$ to $\omega$.  We say that
$\phi$ is \emph{$m$-reducible to $\psi$}, written $\phi\leq_m\psi$,
if there exists a total computable function $g$ with $\phi = \psi\circ g$.
(For strictly partial functions, this includes the requirement that
$(\forall x)[ \phi(x)\converges \iff \psi(g(x))\converges]$.)

If the $m$-reduction $g$ is injective, we say that $\phi$ is
\emph{$1$-reducible to $\psi$}, written $\phi\leq_1\psi$.
\end{defn}

This definition already exists in the literature on computability,
having been presented as part of the theory of numberings studied
by research groups in Novosibirsk and elsewhere.  (See e.g.\
\cite{E77}, \cite{Ebk} or \cite[p.\ 477]{E99} for the notion of reducibility on numberings.)
It is appropriate here as an example of our approach in generalizing
reducibilities on sets to reducibilities on functions, for which reason
we feel justified in calling it $m$-reducibility.
For subsets $A,B\subseteq\omega$, it is quickly seen that
$A\leq_m B$ iff $\chi_A\leq_m\chi_B$, where these are the characteristic functions
of those sets; similarly for $A\leq_1 B$.  The analogue of Myhill's Theorem
for functions states that if $\phi\leq_1\psi$ and $\psi\leq_1\phi$,
then in fact there is a computable permutation $h$ of $\omega$
with $\phi = \psi\circ h$ (and hence $\psi=\phi\circ (h^{-1})$),
in which case we would call $\phi$ and $\psi$ \emph{computably isomorphic}.
The proof is exactly the same as that of the original theorem of Myhill
(see \cite[Thm.\ I.5.4]{S87}), and it makes no difference whether
$\phi$ and $\psi$ are both total or not.

\begin{defn}
\label{defn:completeness}
A function $\psi$ is \emph{$m$-complete} for a class $\Gamma$ of functions
if $\psi\in\Gamma$ and, for every $\phi\in\Gamma$, we have $\phi\leq_m\psi$.
We define $1$-completeness similarly, but require that $\phi\leq_1\psi$.
\end{defn}

For example, the universal Turing function $\psi(\la e,x\ra)=\phi_e(x)$
is $1$-complete partial computable, i.e.\ $1$-complete for the class
of all \emph{partial} computable functions.  (For each single $\phi_e$ in the class,
the function $x\mapsto \la e,x\ra$ is a $1$-reduction.)
A more surprising result is that there does exist a function $h$
which is $1$-complete total computable:  let $h(\la n,m\ra)=m$,
so that, for every total computable $f$, the function
$g(n)=\la n,f(n)\ra$ is a $1$-reduction from $f$ to $h$.

The notion of $m$-reducibility for sets has small irritating features,
particularly the status of the sets $\emptyset$ and $\omega$.
Intuitively, the complexity of each of these is as simple as possible,
yet they are $m$-incomparable to each other.  (Also, no $S\neq\emptyset$
has $S\leq_m\emptyset$ and no $S\neq\omega$ has $S\leq_m \omega$;
intuitively this is reasonable, but it is still strange to have two $m$-degrees
containing only a single set each.)  The same problem is magnified for
$m$-reducibility on functions.  Clearly, if $\phi\leq_m\psi$, then
$\rg{\phi}\subseteq\rg{\psi}$.  It follows that every total
constant function forms an $m$-degree all by itself.
Moreover the function $\phi(x)=2x$ is $m$-incomparable
with $\psi(x)=2x+1$, even though these seem to have
very similar complexity; and assorted other pathologies can be found.
To address these issues, we offer the following adaptation of $m$-reducibility.

\begin{defn}
\label{defn:modified}
The \emph{join} of two partial functions $\phi$ and $\psi$
is the function which splices them together:
$$ (\phi\oplus\psi)(x)=\left\{\begin{array}{cl}
\phi(\frac{x}2), & \text{if $x$ is even;}\\
\psi(\frac{x-1}2), & \text{if $x$ is odd.}
\end{array}\right.$$

A partial function $\phi$ is \emph{augmented $m$-reducible} to another
partial function $\psi$, written $\phi\leq_a\psi$,
if $\phi\leq_m (\iota\oplus \psi)$,
where $\iota(0)\diverges$ and $\iota(x+1)=x$.
\end{defn}

The intention here is that, for any computable subset $S\subseteq\omega$
such that $\phi\res S$ and $(S\cap\dom{\phi})$ are both computable,
one can define the $m$-reduction $g$ from $\phi$ to $\iota\oplus\psi$ by choosing
$g(x)=2+2\phi(x)$ for $x\in S$, with $g(x)=0$ if $\phi(x)$ is known to diverge.
On $(\omega-S)$, $g$ must actually serve as a (computable) $m$-reduction
from $\phi\res(\omega-S)$ to $\psi$.
Under this definition, there is a single $a$-degree consisting of all partial computable
functions with computable domains.  Thus, many pathologies regarding
functions with seemingly similar complexity but distinct domains are avoided.
(Some remain.  A noncomputable function $\phi$ will generally be $a$-incomparable
to $(1+\phi)$, for instance.)

\begin{prop}
\label{prop:pcdegree}
The $a$-degree of the empty function $\lambda$ is the least $a$-degree.
Moreover, a function $\phi$ belongs to this $a$-degree
iff $\phi$ is partial computable and $\dom{\phi}$ is computable.
\end{prop}
\begin{pf}
For every partial function $\phi$, we have $\lambda\leq_a\phi$,
since the constant function $0$
serves as an $m$-reduction from $\lambda$ to $(\iota\oplus\phi)$.
For the forwards direction of the equivalence, let $g$ be an $m$-reduction from $\phi$
to $(\iota\oplus\lambda)$.  Then $x\in\dom{\phi}$ iff $g(x)$ is nonzero and even,
in which case $\phi(x)=(\iota\oplus\lambda)(g(x))=\frac{g(x)}{2}-1$,
which is computable.  For the converse, we define an $m$-reduction $h$
from $\phi$ to $(\iota\oplus\lambda)$ by:
$$ h(x)=\left\{\begin{array}{cl}
2\cdot (\phi(x)+1), & \text{if~}x\in\dom{\phi}\\
0, & \text{if not.}
\end{array}\right.$$
\qed\end{pf}

For characteristic functions of sets $A$ and $B$,
we have $\chi_A\leq_a\chi_B$ iff $A\leq_m (B\oplus\emptyset\oplus\omega)$.
We will not use this concept for sets, but we suggest writing
$A\leq_a B$ whenever $A\leq_m (B\oplus\emptyset\oplus\omega)$.
Under this definition, the computable sets (including
$\emptyset$ and $\omega$) form the least $a$-degree of sets.
The definition of $1$-reducibility for functions could be adapted
in the same way to produce a notion of \emph{augmented $1$-reducibility}
for functions, and likewise for sets, but certain questions arise about
the best way to adapt the definition, and we will not address them here.

We also consider reducibilities intermediate between $m$-reducibility
and Turing reducibility, again by analogy to such reducibilities on sets.

\begin{defn}
\label{defn:tt}
Let $\alpha$ and $\beta$ be total functions.  We say that
$\alpha$ is \emph{bounded-Turing reducible} to $\beta$,
or \emph{weak truth-table reducible} to $\beta$, if there
exists a Turing reduction $\Phi$ of $\alpha$ to $\beta$ and
a computable total function $f$ such that, in computing
each value $\alpha(x)$, the reduction $\Phi$ only asks the $\beta$-oracle
for values $\alpha(y)$ with $y < f(x)$.  Thus, for each $x$,
$\alpha(x)=\Phi^{\beta~\res\!~f(x)}(x)$.  Bowing to the two
distinct terminologies that exist for this notion on sets, we use two
notations for this concept:
$$ \alpha\leq_{\bT}\beta\hspace{5mm}
\text{and}\hspace{5mm}\alpha\leq_\wtt\beta.$$

If we use $D_e$ to denote the finite set with strong index $e$
(i.e.\ the index tells the size of $D_e$ and all of its elements),
then we say that $\alpha$ is \emph{truth-table reducible} to $\beta$, written
$\alpha\leq_\tt\beta$,
if there exist total computable functions $f$ and $g$ such that,
for every input $x$ to $\alpha$, we have $\alpha(x)=g(x,\beta\res
D_{f(x)})$.  This is different from the related reducibilities on
enumerations described by Degtev \cite{Degtev} --- in particular, it
is important for what follows in the present paper that $\alpha$ and
$\beta$ can have different ranges.

Finally, if $\alpha\leq_\tt\beta$ via $f$ and $g$ as above
and there exists some $k\in\omega$ such that $|D_{f(x)}|\leq k$
for every $x\in\omega$, then we say that $\alpha$ is
\emph{bounded truth-table reducible} to $\beta$
\emph{with norm $k$}, and write $\alpha\leq_\kbtt\beta$.
For $\alpha$ to be \emph{bounded truth-table reducible}
to $\beta$ (with no norm stated) simply means that
such a $k$ exists, and is written $\alpha\leq_\btt\beta$.
\end{defn}
It should be noted that, as with sets, the relation $\leq_\kbtt$
on function fails to be transitive, for $k>1$.  In general,
if $\alpha\leq_\normbtt{j}\beta$ and $\beta\leq_\kbtt\gamma$,
then $\alpha\leq_\normbtt{(jk)}\gamma$.

As mentioned, functions and their graphs have always
been conflated for purposes of Turing-reducibility.
For these finer reducibilities, the conflations no longer apply.
\begin{prop}
\label{prop:compbdd}
A total function $h$ is truth-table equivalent to
(the characteristic function of) its own graph
iff there exists a computable function $b$ such that, for every
$x$, we have $h(x)\leq b(x)$.  (In this case, $h$ is said to be
\emph{computably bounded}.)
\end{prop}
\begin{pf}
Let $G\subset\omega^2$ be the graph of $h$, and suppose
first that $h(x)\leq b(x)$ for all $x$.  Then, with a $G$-oracle,
a Turing machine on input $x$ can simply ask which pairs $(x,n)$
with $n\leq b(x)$ lie in $G$.  So we have stated in advance which
oracle questions will be asked, and by assumption there will be exactly one
positive answer, which will be the pair $(x,y)$ with $y=h(x)$.
Thus $h\leq_\tt G$, since we can also say in advance exactly
what answer the machine will give in response to each
possible set of oracle values.  On the other hand, to determine
whether $(x,y)\in G$, an oracle Turing machine only
needs to ask an $h$-oracle one question:  the value of $h(x)$.
Thus $\chi_G\leq_\tt h$ as well.  This latter reduction
is actually a bounded truth-table reduction of norm $1$,
under Definition \ref{defn:tt}, and holds even without
the assumption of computable boundedness of $h$.

For the forwards direction, suppose $h\equiv_\tt \chi_G$.
Then the computation of $h$, on input $x$, asks for the value $\chi_G(m,n)$
only for pairs with codes $\la m,n\ra\in D_{f(x)}$,
and outputs $g(x,\chi_G\res D_{f(x)})$, with $f$ and $g$
as in Definition \ref{defn:tt}.
Thus $h(x)$ must be one of the finitely many
values in the set $\Set{g(x,\sigma)}{\sigma\in 2^{|D_{f(x)}|}}$.
Since $f$ and $g$ are computable and total, we may take
$b(x)$ to be the maximum of this set, forcing $h(x)\leq b(x)$.
Thus $h$ is computably bounded.
\qed\end{pf}

In Proposition \ref{prop:compbdd},
\tt-equivalence cannot be replaced by \bT-equivalence.
The following proof of this fact was devised
in a conversation between E.\ Fokina and one of us,
and completes the answer to the question asked at the
beginning of this section.

\begin{lemma}[Fokina-Miller]
\label{lemma:Fokina}
There exists a total function $f$ which is not computably
bounded, yet is \bT-equivalent to (the characteristic function of)
its own graph $G$.
\end{lemma}
\begin{pf}
Let $K=\Set{\la e,x\ra}{\phi_e(x)\converges}$ be the halting set.
Define $f(2x)=\chi_K(\la x,2x+1\ra)$ on the even numbers,
using the characteristic function $\chi_K$ of $K$, and on the odd
numbers, define
$$ f(2x+1)=\left\{\begin{array}{cl}1+\phi_x(2x+1), & \text{if~}\phi_x(2x+1)\converges,\\
0, & \text{if not.}
\end{array}\right.$$
Of course this $f$ is not computable, but it is total, and for each $x$,
the input $(2x+1)$ witnesses that $\phi_x$ is not an upper bound for $f$.
Moreover, to determine $f(2x)$ on even numbers,
we need only ask a $G$-oracle whether $\la 2x,0\ra\in G$.
To determine $f(2x+1)$ on odd numbers, we again ask
the oracle whether $\la 2x,0\ra\in G$.
If so, then $\chi_K(\la x,2x+1\ra)=f(2x)=0$, meaning that
$\phi_x(2x+1)\diverges$, and so we know that $f(2x+1)=0$.  If
$\la 2x,0\ra\notin G$,
then we know that $\chi_K(\la x,2x+1\ra)=f(2x)=1$, so $\la x,2x+1\ra\in K$,
and we simply compute $\phi_x(2x+1)$ (knowing that it must
converge) and add $1$ to get $f(2x+1)$.  In all cases,
therefore, we can compute $f(y)$ by asking a single
question of the $G$-oracle about whether a predetermined
value lies in $G$.  Thus $f\leq_\bT G$, and of course
$G\leq_\onebtt f$.
\qed\end{pf}

\section{Functions Approximable from Above}
\label{sec:AFA}

Having adapted several standard reducibilities on sets
to serve for functions as well, we now perform the same service
for the Ershov hierarchy.  Traditionally this has been
a hierarchy of $\emptyset'$-computable sets,
determined by computable approximations to those sets
and by the number of times the approximations ``change their mind''
about the membership of a given element in the set.
In our investigations of the distance functions on computable graphs,
we found that similar concepts arose, but pertaining
to functions, not to sets.  Therefore,
the following definitions provide total $\Delta^0_2$-functions with
their own Ershov hierarchy, and then add some further structure.

\begin{defn}
\label{defn:approx}
Let $f(x)=\lim_s g(x,s)$ be a total function from $\omega$ to $\omega$,
with the binary function $g$ total and computable.
\begin{itemize}
\item
If there is a total computable function $h$ such that
$$ (\forall x)~|\Set{s}{g(x,s)\neq g(x,s+1)}|\leq h(x),$$
then $f$ is \emph{$\omega$-approximable}.
\item
If the constant function $h(x)=n$ can serve as the $h$ in the previous item,
then $f$ is \emph{$n$-approximable}.
\item
More generally, if $\alpha$ is a computable ordinal and
there is a total computable nonincreasing function $h:\omega^2\to\alpha$ such that
$$ (\forall x\forall s)~[g(x,s)\neq g(x,s+1) \implies h(x,s)\neq h(x,s+1),$$
then $f$ is \emph{$\alpha$-approximable}.
\item
If, for all $x$ and $s$, we have $g(x,s+1)\leq g(x,s)$, then $f$ is \emph{approximable
from above}.  Such functions are also sometimes called \emph{limitwise decreasing}, \emph{semi-computable from above} or \emph{right c.e.}~functions
\item
If, for all $x$ and $s$, we have $g(x,s+1)\geq g(x,s)$, then $f$ is \emph{approximable
from below}.  In the literature, such functions have also been called
\emph{limitwise monotonic}, \emph{limitwise increasing}, and \emph{subcomputable}.
\item
When we combine these definitions, we assume that a single function $g$ satisfies
all of them.  For instance, $f$ is \emph{$3$-approximable from above} if $f$ is the
limit of a computable function $g$ such that,
for all $x$, $|\Set{s}{g(x,s)\neq g(x,s+1)}|\leq 3$
\emph{and} $(\forall s)~g(x,s+1)\leq g(x,s)$.
\item
Following \cite{CLR09}, we define $f$ to be \emph{graph-$\alpha$-c.e.}
if $\alpha$ is a computable ordinal and the graph of $f$ is an $\alpha$-c.e.\ set
in the Ershov hierarchy.
\end{itemize}
\end{defn}

In \cite{CLR09}, the term \emph{$\alpha$-c.e.\ function} was used for the functions
we are calling $\alpha$-approximable.  We prefer our terminology,
since the phrase ``c.e.\ function'' has been used elsewhere for
functions approximable from below.  (For such a function,
the set $\Set{(x,y)}{y\leq f(x)}$ is c.e.  This also explains the use of the term
\emph{subcomputable} for such functions.)

A characteristic function $\chi_A$ is approximable from below iff $A$
is c.e., and approximable from above iff $A$ is co-c.e.  The definitions
of approximability from below and from above may seem dual,
but in fact there are significant distinctions between them.
For an example, contrast the following easy lemma
with the well-known fact that there exist functions which
are approximable from below, but not $\omega$-approximable.

\begin{lemma}
\label{lemma:allomegace}
Every function approximable from above is $\omega$-approximable from above.
\end{lemma}
\begin{pf}
Let $f=\lim_s g$ with $g(x,s+1)\leq g(x,s)$ for all $x$ and $s$.
Then the computable function $h(x)=g(x,0)$ bounds the number
of changes $g$ can make.
\qed\end{pf}

On the other hand, the hierarchy of $n$-approximability
from above does not collapse.  (See also Corollary
\ref{cor:strictomega} below, which uses this lemma
to show non-collapse at the $\omega$ level, as well.)

\begin{lemma}
\label{lemma:nocollapse}
For every $n$, there is a function $f$ which is $(n+1)$-approximable
from above but not $n$-approximable.
\end{lemma}
\begin{pf}
We define a computable function $g(x,s)$ and set $f(x)=\lim_s g(x,s)$.
To begin, $g(x,0)=n+1$ for every $x$.
At stage $s+1$, for each $x$, compute the sequence
$$ \phi_{x,s}(x,0),\phi_{x,s}(x,1),\ldots,\phi_{x,s}(x,t)$$
for the greatest $t\leq s$ such that all these computations
converge.  If $\phi_x(x,t)=g(x,s) >0$, set $g(x,s+1)=g(x,s)-1$;
otherwise set $g(x,s+1)=g(x,s)$.  Thus $f(x)=\lim_s g(x,s)$
is $(n+1)$-approximable from above, but if
$f(x)=\lim_s\phi_e(x,s)$, then $\phi_e(x,s)$ must
have assumed each of the values $(n+1),n,\ldots,1,0$,
and so $\phi_e$ is not an $n$-approximation to $f$.
\qed\end{pf}

The appropriate duality pairs functions approximable from above
with a subclass of the functions approximable from below, as follows.
\begin{defn}
\label{defn:dual}
Suppose that $g$ is computable and total, with $g(x,s+1)\leq g(x,s)$
for all $x$ and $s$, so that $f(x)=\lim_s g(x,s)$ is total and
approximable from above.  The \emph{dual of $g$} is the function
$$h(x,s)= g(x,0)-g(x,s).$$
Thus $j(x)=\lim_s h(x,s)$ is total, approximable from below (by $h$),
and bounded above by $g(x,0)$.  Moreover, $g$ is an $\alpha$-approximation
for $f$ iff $h$ is an $\alpha$-approximation for $j$.

Conversely, let $j$ be any function which is approximable from below
via $h(x,s)$ and \emph{computably bounded}:  that is, $j$ is such that
there exists a computable total
function $b$ with $j(x)\leq b(x)$ for all $x$.  Then the function
$g(x,s)=b(x)-h(x,s)$ is the \emph{dual of $h$ and $b$}.
\end{defn}

It is natural to call $j$ the dual of $f$, but in fact $j$ depends
on the choice of the approximation $g$:  two different approximations
$g$ and $\gtilde$ will often yield two different duals, though these
two duals always differ by a computable function, namely $(g(x,0)-\gtilde(x,0))$.
The dual of a function approximable from below also depends
on the choice of computable bound.  Nevertheless, it will be clear
from our results below that the class of computably bounded functions
approximable from below is the natural dual for the class of functions
approximable from above.  The computable upper bound in the former class
is the obvious counterpart of the built-in computable lower bound of $0$
for the latter class.

Functions approximable from below have seen wide usage in computable
model theory, for example in \cite{CCHM1,CCHM2,HMP05,KhisO,KhisHRM,KNS97}.
Our interest in functions
approximable from above arose from our investigations into the distance function on
a computable graph.  To our knowledge, this is the first significant use
of such functions in computable model theory, although, as we will mention,
they arise implicitly in the study of effectively algebraic structures and in certain
other contexts.  The best-known example of a function approximable from above
does not come from computable model theory at all:  it is the function of
Kolmogorov complexity (for any fixed universal machine), mapping
each finite binary string (coded as a natural number) to the shortest program
which the fixed machine can use to output that string.

It was a theorem of Khoussainov, Nies, and Shore in \cite{KNS97}
that there exists a $\Delta^0_2$ set which is not the range
of any function approximable from below.  The following theorem
contrasts with that result, giving a very concrete distinction
between approximability from above and from below.

\begin{thm}\label{thm:ranges}The range of every approximable
function is the range of some function which is $2$-approximable from above.
Indeed, the ranges of the $2$-approximable-from-above functions
are precisely the $\Sigma^0_2$ sets.
\end{thm}
\begin{pf}
We prove the stronger statement.
Being in the range of a $2$-approximable-from-above function
is clearly a $\Sigma^0_2$ condition.
For the converse, let $S\in\Sigma^0_2$ have computable
$1$-reduction $p$ to the $\Sigma^0_2$-complete set $\Fin$, so that
$$ (\forall x)[x\in S \iff |W_{p(x)}|<\infty].$$
We will assume that at each stage $s$, there is exactly
one $x$ such that $W_{p(x),s+1}\neq W_{p(x),s}$,
and also that for every $x$, $W_{p(x)}\neq\emptyset$;
both of these conditions are readily arranged.
Fix the least $x_0\in S$.  At stage $0$ we define nothing.
At stage $s+1$, for each $x < s$, define
$$g(x,s+1) = \left\{\begin{array}{cl}
x_0, & \text{if~}g(x,s)=x_0\\
x_0, & \text{if~}W_{g(x,s),s+1}\neq W_{g(x,s),s}\\
g(x,s), & \text{otherwise.}
\end{array}\right.$$
Then, for the unique $y$ such that $W_{p(y),s+1}\neq W_{p(y),s}$, let
$$g(s,0)=g(s,1)=\cdots =g(s,s+1)=\left\{\begin{array}{cl}p(y), & \text{if~}p(y)\geq x_0\\
x_0, & \text{if not.}
\end{array}\right.$$
This defines $g$ effectively on all of $\omega\times\omega$,
and for every $x$, $g(x,s)$ is either $x_0$ for all $s$,
or $p(y)$ for all $s$ (where $y$ was chosen at stage $x+1$),
or else $p(y)$ for $s=0,1,\ldots,n$ and then $x_0$ for all
$s>n$.  This last holds iff $p(y)>x_0$ and $W_{p(x)}$ received a new
element at some stage $n+1>x+1$.  So
clearly $g$ has a limit $f(x)=\lim_s g(x,s)$
and approximates that limit from above, with at most one change.
Moreover, if $y\in S$, then $p(y)\in\Fin$, and so when
(the nonempty set) $W_{p(y)}$ receives its last element,
say at stage $x+1$, then $x$ will have $g(x,s)=p(y)$ for all $s>x$,
making $S\subseteq\rg{f}$.  Conversely, if $x\notin S$,
then $p(y)\notin\Fin$, so every $x$ which ever had $g(x,s)=p(y)$
will eventually get changed and will have $f(x)=x_0$;
thus $\rg{f}\subseteq S$.
\qed\end{pf}

%


\begin{cor}There is a function $2$-approximable from above, the range of which
is not the range of any function approximable from below.
\end{cor}
\begin{pf}
This is immediate from Theorem \ref{thm:ranges} in conjunction
with a result in \cite{KNS97} giving the existence of a $\Delta^0_2$ set
which is not the range of any function approximable from below.
\qed\end{pf}

\begin{cor}
\label{cor:Sigma2}
There exists a function that is $2$-approximable from above
whose range is $\Sigma^0_2$-complete.
\qed\end{cor}


\begin{thm}
\label{thm:countdown}
Every $\omega$-approximable function $f$ is \bT-reducible
to some function approximable from above.
(This approximation from above is known as the
\emph{countdown function} for $f$.)
\end{thm}

\begin{proof}  Let $f(x)$ be an $\omega$-approximable function,
approximated by $g(x,s)$, with computable function $h(x)$ bounding
the number of mind changes of $f(x,s)$.  Set $c(x,0)=h(x)$.
Let $c(x, s+1)=c(x,s)$ unless $g(x,s)\neq g(x, s+1)$, and
in that case set $c(x,s+1)=c(x,s)-1$.  Then $\lim_s c(x,s)$
is approximable from above and $f(x)$ is computable from this limit,
since $f(x)=f(x,t)$ for each $t$ with $c(x,t)=\lim_s c(x,s)$.
	
Our reasons for referring to this $c$ as the \emph{countdown function}
for $f$ (or, strictly speaking, for $g$ and $h$, since $c$ does
depend on the approximation and the computable bound) are clear.
It is important to distinguish the countdown function
$c(x,s)$, which is computable, from its limit $\lim_s c(x,s)$,
which in general is not computable (and was just shown to Turing-compute $f$).
Indeed, we have $f\leq_{\bT} \lim c$, since the only value
of the limit required to compute $f(x)$ is $\lim_s c(x,s)$.
\end{proof}

On the other hand, this is not in general a truth-table
reduction.  For that, one would need to
predict in advance what the value of $f(x)$ will be
for every possible value of $\lim_s c(x,s)$ between
$0$ and $c(x,0)$.  Without knowing $\lim_s c(x,s)$ in advance,
one cannot be sure for how many values of $s$
we may need to compute $g(x,s)$ to determine these answers.

The limit of $c$ is not in general Turing-reducible to $f$.
However, if $g$ is either an approximation from above or
an $\omega$-approximation from below, then $f\equiv_T \lim_s c(\cdot,s)$,
and indeed $f\equiv_{\bT}\lim_s c(\cdot,s)$, since the computation
of $\lim_s c(x,s)$ only requires us to ask the oracle for the value $f(x)$.
(Once an approximation from above or from below abandons a value,
it cannot later return to that value, and so, once $f(x)=g(x,t)$,
we know that $c(x,t)=\lim_s c(x,s)$.)  However, even for
approximations from above and $\omega$-approximations from below,
we have in general that $\lim c\not\leq_\tt f$, since we cannot
determine the final value of the countdown without actually
knowing the value $f(x)$.



\begin{thm}
\label{thm:1complete}
There is a function $f$ that is $1$-complete
within the class of all functions approximable from above.
\end{thm}
\begin{pf}
We construct $f$ by constructing a computable function $g$ approximating
$f$ from above.  At stage $0$, $g$ is undefined on all inputs.

At stage $s+1$, find the least pair $k=\la e,x\ra$ (if any) such that
$\phi_{e,s}(x,0)\converges$ and $n_k$ is undefined.  Let $n_k$ be the least
element such that $g(n_k,0)$ was undefined as of stage $s$, and set
$$g(n_k,0)=g(n_k,1)=\cdots =g(n_k,s+1)=\phi_{e,s}(x,0).$$

Then (whether or not such a $k$ existed), for each $j=\la e,x\ra$
such that $g(n_j,0)$ was defined by stage $s$, we consider the sequence
$\phi_{e,s}(x,0),\ldots,\phi_{e,s}(x,t)$, for the greatest $t\leq s$
such that all these computations converge.  If this finite sequence
is nonincreasing, we set $g(n_j,s+1)=\phi_e(x,t)$.  Otherwise
(that is, if $\phi_e(x,t'+1)>\phi_e(x,t')$ for some $t'<t$),
we set $g(n_j,s+1)=g(n_j,s)$.

This completes the construction of $g$.  Clearly, every $n$ is chosen
at some stage to be $n_k$ for some $k$, and subsequently
$g(n_k,s)$ is defined for each $s$, so $g$ is total and computable.
Moreover, by construction, $g(n_k,s+1)\leq g(n_k,s)$ for every $k$ and $s$.
So the function $f(n)=\lim_s g(n,s)$ is approximable from above.

Now let $h$ be any function which is approximable from above,
say by $h(x)=\lim_s\phi_e(x,s)$.  Then, for every $x$,
$\phi_e(x,0)$ converges at some finite stage, and so some
$n_k$ with $k=\la e,x\ra$ is eventually chosen.  The function
$d_e$ mapping $x$ to this $n_{\la e,x\ra}$ (for this fixed $e$)
is computable (since we can simply wait until $n_{\la e,x\ra}$ is defined)
and total, and also 1-1, since $n_j\neq n_k$ for $j\neq k$.
Now since $\phi_e$ approximates $h$ from above, the sequence
$$ \phi_e(x,0),\phi_e(x,1),\phi_e(x,2),\ldots,\phi_e(x,t),\ldots $$
is infinite and nonincreasing.  So the sequence
$$ g(n_k,0),g(n_k,1),g(n_k,2),\ldots,g(n_k,s),\ldots$$
is exactly the same sequence, by construction, except that
numbers which occur finitely often in one sequence might occur
a different finite number of times in the other sequence.
This shows that
$$ f(d_e(x)) = f(n_k) = \lim_s g(n_k,s) = \lim_t \phi_e(x,t) = h(x),$$
so $f\circ d_e = h$, proving $h\leq_1 f$ via $d_e$.
Indeed, therefore, the $1$-reduction may be found uniformly
in the index of a computable approximation to $h$ from above.
\qed\end{pf}

Theorem \ref{thm:1complete} gives an easy proof of a result
which we could have shown by the method from Lemma
\ref{lemma:nocollapse}.

\begin{cor}
\label{cor:strictomega}
There exists a function which is $\omega$-approximable from above,
but (for every $n\in\omega$) is not $n$-approximable.
\end{cor}
\begin{pf}
If the $1$-complete function $f$ from Theorem \ref{thm:1complete}
were $n$-approximable, say via $g(x,s)$, then since every function
approximable from above has a $1$-reduction $h$ to $f$, we would have that every such function is $n$-approximable (via $g(h(x),s)$).  This contradicts Lemma \ref{lemma:nocollapse}.
\qed\end{pf}

Theorem \ref{thm:1complete} stands in contrast to the following results.

\begin{thm}
\label{thm:no1complete}
No function is $m$-complete for the class of functions approximable from below.
\end{thm}
\begin{pf}
Let $f$ be any function with a computable approximation $g(x,s)$ of $f$ from below.
We build a computable approximation $h$ from below
to a function $j$ which cannot be $m$-reducible to $f$.  Define $h(\la e,x\ra,0)=0$
for every $e$ and $x$.

At stage $s+1$, if $\phi_{e,s}(\la e,x\ra)\diverges$, set $h(\la e,x\ra,s+1)=0$.
Otherwise, set $h(\la e,x\ra,s+1)=1+g(\phi_e(\la e,x\ra),s+1)$.
By hypothesis on $g$, this $h$ is clearly nondecreasing in $s$,
and since $\lim_s g(\phi_e(\la e,x\ra,s))=f(\la e,x\ra)$ must exist,
we see that
$$j(\la e,x\ra) =\lim_s h(\la e,x\ra,s) =1+\lim_s g(\phi_e(\la e,x\ra),s)
=1+f(\phi_e(\la e,x\ra)).$$
However, this shows that either $j\neq f\circ\phi_e$ or else $\phi_e$
is not total, so $j\not\leq_m f$.
\qed\end{pf}

\begin{thm}
\label{thm:no_mcomplete_for_omegace}
For every function $f$ approximable from above, there exists
a function $g$ which is $2$-approximable from below and has $g\not\leq_m f$.
\end{thm}

	\begin{proof}	
	Given a function $f$ that is approximable from above, we construct an $\omega$-approximable function, $g$, that $f$ fails to $m$-compute.  To achieve this, we assume that $f(x)=\lim_s \varphi_f(x,s)$, and consider possible witnesses for an $m$-reduction, diagonalizing against them.  We must meet the following requirement for each $e\in \omega$: $$R_e: ~ ~ \varphi_e \textrm{ is total }\implies \exists x ~ f(\varphi_e(x))\neq g(x).$$
	
	Set $g(x,0)=0$ for each $x\in \omega$.  If
        $\varphi_e(e)\downarrow = y_e$ at stage $s$, set
        $g(e,t)=f(y_e,s)+1$ for all $t\geq s$.  For $t>s$, and any
        $y$, we have $f(y,t)<f(y,s)$, so once $R_e$ receives attention, it is forever satisfied.  Furthermore, $f(e,0)$ serves as a computable upper bound for $g(e,s)$ so $g$ is computably bounded, and the number of mind changes is at most 2.

	\end{proof}

The following result suggests the insufficiency of Turing reducibility,
and even of bounded Turing reducibility,
to classify functions approximable from above.
\begin{prop}
\label{prop:cedegree}
Every function approximable from above is bounded-Turing equivalent
to (the characteristic function of) a computably enumerable set,
and every c.e.\ set is \bT-equivalent to a function approximable from above.
\end{prop}
\begin{pf}
Let $f=\lim_s g(x,s)$ be an approximation of $f$ from above.
Define the c.e.\ set
$$ V_f = \Set{\la x,n\ra}{\exists^{>n}s~[g(x,s+1)\neq g(x,s)]}.$$
That is, $\la x,n\ra\in V_f$ if and only if the approximation to $f$
changes its mind more than $n$ times.  Clearly $V_f\leq_T f$,
and the computation deciding whether $\la x,n\ra\in V_f$
requires us only to ask the oracle for the value of $f(x)$.
Conversely, to compute $f(x)$ from a $V_f$-oracle,
we first compute $g(x,0)$, and then ask the oracle which of the values
$\la x,0\ra,\ldots,\la x,g(x,0)-1\ra$ lies in $V_f$;
the collective answer tells us exactly how many times the
approximation will change its mind, and with this information
we simply compute $g(x,s)$ until we have seen that many mind changes.
Both of these are bounded Turing reductions.  (Neither, however,
is a truth-table reduction.  In fact, the characteristic function of $V_f$
is \tt-equivalent to the limit of the countdown function for $f$,
which is not in general \tt-equivalent to $f$.)

For the second statement, note that every c.e.\ set is \bT-equivalent
(indeed \tt-equivalence with norm $1$)
to the characteristic function of its complement.
\qed\end{pf}

\begin{cor}
There is a $2$-approximable function that is not
Turing equivalent to any function approximable from above.
\end{cor}
\begin{pf}
Let $S$ be a d.c.e.\ set which is not of c.e.\ degree.
(See, e.g., \cite{EHK} for such a construction).  Then the characteristic function $\chi_S$
is $2$-approximable, and Proposition \ref{prop:cedegree}
completes the result.
\qed\end{pf}

\section{The Distance Function in Computable Graphs}
\label{sec:graphs}

With the results of the preceding sections completed, we may now address
the intended topic of this paper:  the distance function on a computable graph.
Computable graphs are defined by the standard computable-model-theoretic definition.
\begin{defn}
\label{defn:computablegraph}
A structure $\M$ in a finite signature is \emph{computable}
if it has an initial segment of $\omega$ as its domain
and all functions and relations on $\M$ are computable
when viewed as functions and relations (of appropriate
arities) on that domain.

Therefore, a symmetric irreflexive graph $G$ is \emph{computable}
if its domain is either $\omega$ or a finite initial segment thereof,
and if there is an algorithm which decides,
for arbitrary $x,y\in G$, whether $G$ contains an edge
between $x$ and $y$ or not (i.e.\ if the algorithm computes the
entries of the adjacency matrix).
\end{defn}

The \emph{distance function} $d$ on a graph $G$
maps each pair $(x,y)\in G^2$ to the length
of the shortest path from $x$ to $y$.
Assuming $G$ is connected, such a path must exist.
By definition $d(x,x)=0$, and this is quickly
seen to be a metric on $G$.  Moreover, if
$G$ is a computable graph, then the distance
function on $G$ is approximable from above, since
for any $x$ and $y$, we can simply search
for the shortest path from $x$ to $y$.
Formally, letting $G_s$ be the induced subgraph of $G$
on the vertices $\{ 0,\ldots,s\}$, we find
the least $t$ for which $G_t$ contains $x$, $y$,
and a path between them, and let $g(x,y,0)$
be the length of the shortest such path in $G_t$.
Then we define $g(x,y,s+1)$ to be the minimum of $g(x,y,s)$
and the length of the shortest path (if any) between
$x$ and $y$ in $G_s$.  This sequence is decreasing in $s$,
with limit $d(x,y)$.  In the language of computable
model theory, we say that for every connected computable graph $G$,
the distance function is \emph{intrinsically approximable from above},
since it is approximable from above in every computable graph
isomorphic to $G$.  (See Definition \ref{defn:intrinsic} below.)

We will now describe a graph construction that enables us to encode
arbitrary functions approximable from above into the distance function
on a computable graph.  The definition is best understood by looking
at the subsequent diagram.

\begin{defn}
\label{defn:spokes}
Let $\sigma\in\omega^{<\omega}$ be any strictly decreasing nonempty finite string.
A \emph{spoke} of type $\sigma$ in a graph consists of the following.
\begin{itemize}
\item
A \emph{center node} $u$, 
which will be part of every spoke; and
\item
two more nodes $a$ (adjacent to $u$) and $b$ (not adjacent to $u$, nor to $a$),
which belong to this particular spoke; and
\item
for each $n<|\sigma|$, a path from $a$ to $b$ consisting of $(1+\sigma(n))$
more nodes (hence of length $2+\sigma(n)$); and
\item
two more paths from $a$ to $b$, of length $2+\sigma(0)$,
in addition to the one already built for $n=0$ in the preceding instruction.
\end{itemize}
These paths do not intersect each other, except at $a$ and $b$, and if $x$ and $y$
are nodes from two distinct paths, then $x$ and $y$ are not adjacent to each other,
nor to $u$.

For any function $\Gamma$ mapping each $m\in\omega$
to a finite decreasing nonempty string
$\sigma_m\in\omega^{<\omega}$, the \emph{standard graph of type $\Gamma$}
consists of a center node $u$ and, for each $m$,
a spoke of type $\sigma_m$.  (If $\Gamma$ is not injective, then there are just as many
spokes of type $\sigma$ as there are elements in $\Gamma^{-1}(\sigma)$.)
\end{defn}
Here is a picture of three spokes of such a graph, with $\sigma=\la 2\ra$,
$\sigma'=\la 2,1\ra$, and $\sigma''=\la 2,1,0\ra$.

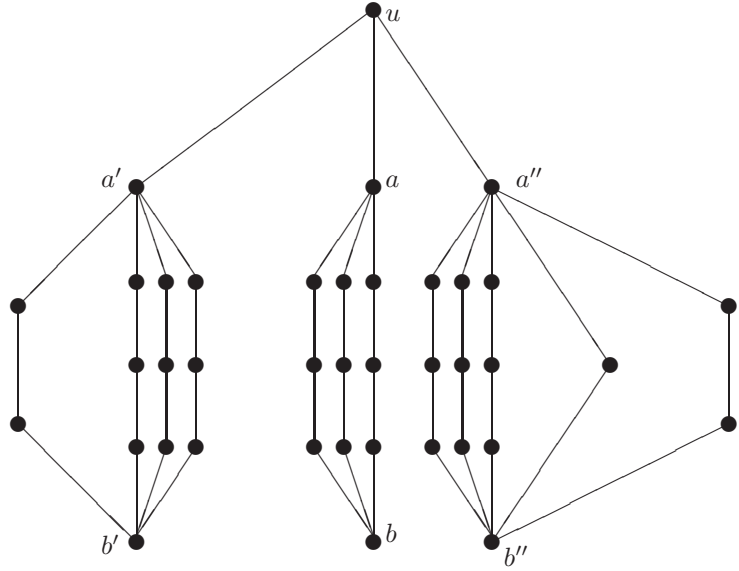
\setlength{\unitlength}{0.62in}
\begin{picture}(7.5,5)(-0.1,0)

\put(4,5){\circle*{0.125}}
\put(4.1,4.9){$u$}

\put(4,3.5){\circle*{0.125}}
\put(4.1,3.5){$a$}
\put(4,2.7){\circle*{0.125}}
\put(4,2.0){\circle*{0.125}}
\put(4,1.3){\circle*{0.125}}
\put(4,0.5){\circle*{0.125}}
\put(4.1,0.5){$b$}
\put(4,5){\line(0,-1){1.5}}
\put(4,3.5){\line(0,-1){3}}

\put(4,3.5){\line(-2,-3){0.5}}
\put(3.5,2.7){\circle*{0.125}}
\put(3.5,2.0){\circle*{0.125}}
\put(3.5,1.3){\circle*{0.125}}
\put(3.5,2.7){\line(0,-1){1.4}}
\put(3.5,1.3){\line(2,-3){0.5}}

\put(4,3.5){\line(-1,-3){0.25}}
\put(3.75,2.7){\circle*{0.125}}
\put(3.75,2.0){\circle*{0.125}}
\put(3.75,1.3){\circle*{0.125}}
\put(3.75,2.7){\line(0,-1){1.4}}
\put(3.75,1.3){\line(1,-3){0.25}}

\put(4,5){\line(-4,-3){2}}
\put(2,3.5){\circle*{0.125}}
\put(1.7,3.5){$a'$}
\put(2,2.7){\circle*{0.125}}
\put(2,2.0){\circle*{0.125}}
\put(2,1.3){\circle*{0.125}}
\put(2,0.5){\circle*{0.125}}
\put(1.7,0.4){$b'$}
\put(2,3.5){\line(0,-1){3.0}}

\put(2,3.5){\line(2,-3){0.5}}
\put(2.5,2.7){\circle*{0.125}}
\put(2.5,2.0){\circle*{0.125}}
\put(2.5,1.3){\circle*{0.125}}
\put(2.5,2.7){\line(0,-1){1.4}}
\put(2.5,1.3){\line(-2,-3){0.5}}

\put(2,3.5){\line(1,-3){0.25}}
\put(2.25,2.7){\circle*{0.125}}
\put(2.25,2.0){\circle*{0.125}}
\put(2.25,1.3){\circle*{0.125}}
\put(2.25,2.7){\line(0,-1){1.4}}
\put(2.25,1.3){\line(-1,-3){0.25}}

\put(1,2.5){\circle*{0.125}}
\put(1,1.5){\circle*{0.125}}
\put(2,3.5){\line(-1,-1){1}}
\put(1,2.5){\line(0,-1){1.0}}
\put(1,1.5){\line(1,-1){1}}

\put(4,5){\line(2,-3){1}}
\put(5,3.5){\circle*{0.125}}
\put(5.2,3.5){$a''$}
\put(5,2.7){\circle*{0.125}}
\put(5,2.0){\circle*{0.125}}
\put(5,1.3){\circle*{0.125}}
\put(5,0.5){\circle*{0.125}}
\put(5.1,0.3){$b''$}
\put(5,3.5){\line(0,-1){3.0}}

\put(5,3.5){\line(-2,-3){0.5}}
\put(4.5,2.7){\circle*{0.125}}
\put(4.5,2.0){\circle*{0.125}}
\put(4.5,1.3){\circle*{0.125}}
\put(4.5,2.7){\line(0,-1){1.4}}
\put(4.5,1.3){\line(2,-3){0.5}}

\put(5,3.5){\line(-1,-3){0.25}}
\put(4.75,2.7){\circle*{0.125}}
\put(4.75,2.0){\circle*{0.125}}
\put(4.75,1.3){\circle*{0.125}}
\put(4.75,2.7){\line(0,-1){1.4}}
\put(4.75,1.3){\line(1,-3){0.25}}

\put(7,2.5){\circle*{0.125}}
\put(7,1.5){\circle*{0.125}}
\put(5,3.5){\line(2,-1){2}}
\put(7,2.5){\line(0,-1){1.0}}
\put(7,1.5){\line(-2,-1){2}}

\put(6,2.0){\circle*{0.125}}
\put(5,3.5){\line(2,-3){1}}
\put(6,2.0){\line(-2,-3){1}}

\end{picture}

If $f$ is approximated from above by a computable $g$,
and $g(n,0)=2$, for instance, then we build a spoke
between $a_n$ and $b_n$ of type $\la 2\ra$
(the type between $a$ and $b$ in the diagram above).
If, for some $s>0$, we find $g(n,s)=1$,
then we add a path of length $3$ to that spoke between $a_n$ and $b_n$
changing it to type $\la 2,1\ra$,
so that it looks like the spoke between $a'$ and $b'$ above.
If it turns out that $g(n,t)=0$ for some $t$, then we add one more
path, as in the spoke between $a''$ and $b''$,
leaving a spoke of type $\la 2,1,0\ra$.  In all cases,
the distance from the node $a_n$ to the node $b_n$ of this spoke
turns out to be $\lim_s g(n,s)$, which is to say $f(n)$.  So
an arbitrary function approximable from above can be coded into
a distance function in this way.  However,
the distance function for the entire graph needs to do more than
just to determine the length of each single spoke,
and complications will arise when we apply our strong
reducibilities on functions.

If $\Gamma$ is computable, then of course we have a computable presentation of
the standard graph of type $\Gamma$.  However, we will usually be interested in the
situation where $\Gamma=\lim_s \Gamma_s$ and $\Gamma_s$ is computable uniformly in $s$,
with each $\Gamma_s(m)$ being an initial segment of $\Gamma_{s+1}(m)$, such that
$\Gamma(m)=\cup_s \Gamma_s(m)$.  Since every string $\Gamma_s(m)$ is strictly decreasing,
$\Gamma(m)$ will be another finite decreasing string.  Assuming that the
functions $\Gamma_s$ are computable uniformly in $s$, it is clear how
to build a computable presentation of the standard graph of type $\Gamma$,
as the union of nested uniform presentations of the standard graphs of type $\Gamma_s$.

As already noted, the distance function of a computable
connected graph is always approximable from above.
Not every function approximable from above can be a
distance function, however:  for one thing, the range of the distance function $d$
of a connected graph $G$ is always an initial segment of $\omega$, whereas
plenty of functions approximable from above do not have such ranges.
Indeed, for each $x\in G$, the set $\Set{d(x,y)}{y\in G}$
must be an initial segment of $\omega$.
Moreover, symmetry and a triangle inequality must hold of every
distance function.  Nevertheless, we do have the following proposition,
as well as the stronger version given in the subsequent theorem,
which is the result we prove.

\begin{prop}
\label{prop:singledegree}
Every function which is approximable from above is \btt-equivalent,
with norm $2$, to the distance function for some computable graph.
\end{prop}

\begin{thm}
\label{thm:singledegree}
For every function $f$ which is approximable from above,
there is a  computable graph $G$ such that
the Turing degree spectrum of its distance function
$d$ is $\{deg(f)\}$.  Indeed, for every computable $H$ isomorphic to $G$,
the distance function $d_H$ on $H$ satisfies
$$d_H\equiv_1 d \leq_{\twobtt} f \leq_{\onebtt} d_H.$$
\end{thm}
\begin{pf}
Write $f=\lim_s g(x,s)$, where $g$ is computable and nonincreasing
in $s$.  For each $x$, let $\Gamma_s(x,s)=\la g(x,s_0),g(x,s_1),\ldots,g(x,s_k)\ra$,
where $s_0=0$ and the subsequent $s_i$ are defined so that:
$$ g(x,s_0)=g(x,s_1-1) > g(x,s_1)= g(x,s_2-1)>g(x,s_2) =\cdots > g(x,s_k)=g(x,s).$$
That is, the $s_i$ are the stages $\leq s$ at which $g(x,s)$ decreases.
The uniformly computable family $\Gamma_s$ then has limit $\Gamma$ Turing-equivalent
to $f$, since $\Gamma(x)$ is a decreasing string whose final value is $f(x)$.
Therefore, there is a computable presentation $\Gtilde$ of the standard
graph of type $\Gamma$.  Our computable graph $G$ is simply
$\Gtilde$ with (for each $n$) an additional $(n+2)$ nodes
which, along with $a_n$, form a loop of length $(n+3)$ containing $a_n$.
This allows us to determine, for arbitrary nodes $a$ adjacent
to the top center node $u_H$ in an arbitrary computable
graph $H$ isomorphic to $G$, the value $n$ for which $a=a_n$.
(The set $\Set{a_n}{n\in\omega}$
is defined by adjacency to the single node $u_H$.
Then, for each $a_n$, the three paths
of equal length between $a_n$ and $b_n$ allow us to identify $b_n$,
so that we will not confuse the loop of length $n+3$
containing $a_n$ with any loop which contains both $a_n$ and $b_n$.
Having picked out $b_n$, we then find the unique loop which
contains $a_n$ but neither $b_n$ nor $u_H$, and the length of
this loop determines the index $n$ for us.)
Thus, with these new loops added, the entire graph $G$ is relatively
computably categorical, having a $\Sigma^0_1$ Scott family
defined using the loops.  So every computable copy $H$ of $G$
is computably isomorphic to $G$.

It follows that the distance function $d_H$ on the arbitrary computable
copy $H$ is $1$-equivalent to the distance function $d$ on $G$.
Indeed, if $h:G\to H$ is a computable isomorphism,
write $h_2:G^2\to H^2$ by $h_2(x,y)=(h(x),h(y))$.
Then $h_s$ is computable and $d_H\circ h_2=d$;
likewise $d_H=d\circ (h^{-1})_2$.

But $f$ allows us to decide the shortest path
from $a_n$ to $b_n$, for every $n$, since $(2+f(n))$ is the length
of that path.  It follows that the distance function $d$ on all of $G$
is computable in $f$.  (This is made explicit in Lemma
\ref{lemma:standarddistance}, which goes through for this
graph with the extra loops, as well as for standard
presentations.)  Conversely, the distance function $d$
allows us to compute $(2+f(n))$ for every $n$,
just by finding $d(a_n,b_n)$.  So this distance
function is Turing-equivalent to $f$, and by the $1$-equivalence
above, the spectrum of the distance function is simply $\{\deg(f)\}$.
More specifically, we have $f\leq_\onebtt d$
(since $f(n)=2+d(a_n,b_n)$) and $d\leq_\twobtt f$,
using the proof of Lemma \ref{lemma:standarddistance}.
\qed\end{pf}

\begin{cor}
\label{cor:strictomegagraph}
There exists a computable graph $G$ such that, for every $n$
and for every computable graph $H\cong G$,
the distance function on $H$ is not $n$-approximable from above.
(So the distance function on this $G$ intrinsically fails to be $n$-AFA.)
\end{cor}
\begin{pf}
By Corollary \ref{cor:strictomega}, there is a function $f$
which is $\omega$-approximable from above, but not
$n$-approximable for any $n$.  Apply
Theorem \ref{thm:singledegree} to this $f$,
and note that if $f$ were $1$-reducible to a function
$n$-approximable from above, then $f$ itself
would be $n$-approximable from above.
\qed\end{pf}

To complete the proof of Theorem \ref{thm:singledegree},
we need the following lemma.

\begin{lemma}
\label{lemma:standarddistance}
Let $G$ be a computable copy of the standard graph
of some type $\Gamma$, and let $u$ be the center of $G$.
Then there is are computable
functions $p_a$ and $p_b$ which map every node $x\in G-\{ u\}$
to the unique node $p_a(x)$ adjacent to $u$ such that
$x$ and $p_a(x)$ are on the same spoke of $G$,
and to the other end point $p_b(x)$ of that spoke.
Moreover, if we define
$$ S=\Set{\la a,p_b(a)\ra\in G^2}{a\text{~is adjacent to~}u},$$
then the distance function $d$ for $G$ satisfies
$$d \leq_\twobtt (d\res S) \leq_\onebtt d.$$
\end{lemma}
\begin{pf}
Given any $x\in G$ with $x\neq u$,
we search for a path in $G$ which goes from $x$
to $u$ without containing $u$ (except as its end point).
The node on this path
which is adjacent to $u$ must be the desired $p_a(x)$,
simply because of the structure of $G$.

We then compute $p_b(x)$ by finding three paths of equal length
from $p_a(x)$ to a common end point, such that none
of these paths contains $u$ or intersects another of the
three paths (except at their end points).  The common end point must then be
$p_b(x)$.  (This is the reason why a spoke of type $\sigma$
has two extra paths of length $\sigma(0)$.  Without those paths,
there could exist computable copies of $G$ in which this
function $p_b$ would not be computable.)

To compare $d$ with $d\res S$, consider any $x,y\in G-\{ u\}$.
If $x$ and $y$ lie on the same spoke (that is, if $p_a(x)=p_a(y)$),
then we can check whether $x$ and $y$
lie on the same path from $a=p_a(x)$ to $b=p_b(x)$.
This gives three cases.
\begin{enumerate}
\item
If $x$ and $y$ lie on the same path between the same
$a$ and $b$, assume
without loss of generality that on this path,
$x$ lies closer to $a$, say with $m$ nodes between them,
and $y$ lies closer to $b$, with $n$ nodes between them.
Then $d(x,y)$ is either the distance between them along
this path, or else $(m+d(a,b)+n)$, whichever is smaller.
So we need only ask the $(d\res S)$-oracle for
the value $d(a,b)$, and then compare these two possibilities.
\item
If $x$ and $y$ lie on different paths within the spoke
between the same $a$ and $b$,
write $d'(a,x)$ and $d'(x,b)$ for the distances
between those nodes along the path through $x$,
and $d'(a,y)$ and $d'(y,b)$ likewise along the path through $y$.
(It is possible that $d(a,x) < d'(a,x)$, if there is a much shorter
separate path from $a$ to $b$.  However, $d'(a,x)$ is computable.)
Then $d(x,y)$ is the least of
the lengths of the following paths from $x$ to $y$:
\begin{align*}
&[d'(x,a)+d'(a,y)]\text{~~(via a path through $a$)};\\
&[d'(x,b)+d'(b,y)]\text{~~(via a path through $b$)};\\
&[d'(x,a)+d(a,b)+d'(b,y)]\text{~~(via a path through $a$,
then $b$)};\\
&[d'(x,b)+d(b,a)+d'(a,y)]\text{~~(via a path through $b$,
then $a$)}.
\end{align*}
We can compute all of these by asking the $(d\res S)$-oracle for $d(a,b)$,
and having done so, we need only take the minimum
of these four values.
\item
If $x$ and $y$ lie on distinct spokes, find the
end points $a_x=p_a(x)$, $a_y=p_a(y)$, $b_x=p_b(x)$,
and $b_y=p_b(y)$ of those spokes.  As in Case (2), we use $d'(x,a_x)$
to denote the distance from $x$ to $a_x$ along the path
from $a_x$ to $a_y$ through $x$, which we can compute;
similarly for $d'(y,a_y)$, etc.

First we compute $d(x,u)$, which is the minimum of
\begin{align*}
d'(x,a_x)&+1,\text{~~(via a path through $a_x$), and}\\
d'(x,b_x)&+d(b_x,a_x)+1,\text{~~(via a path through $b_x$, then $a_x$)}
\end{align*}
This uses the oracle for $d\res S$.  Likewise we compute $d(y,u)$.
Then it is clear from the structure of $G$ that $d(x,y)=d(x,u)+d(u,y)$.
Notice that here in Case (3) we needed to ask two questions
of the $(d\res S)$-oracle:  the values of $d(a_x,b_x)$ and $d(a_y,b_y)$.
\end{enumerate}
Finally, Case (3) showed how to compute $d(x,u)$.
So we have computed $d(x,y)$ for all possible pairs $(x,y)\in G^2$,
proving $d\leq_T (d\res S)$, and the only questions we asked of the
oracle for $d\res S$ were the values $d(p_a(x),p_b(x))$
and $d(p_a(y),p_b(y))$.  Moreover, a close reading of the proof
shows that we could give in advance
a formula for $d(x,y)$ based on these two values, such that
the formula always converges to an answer.  (The formulas
are slightly different in Cases (1), (2), and (3), but we could
distinguish these three cases and choose the correct one for
the pair $(x,y)$ before consulting the oracle at all.)
Thus we have a \texttt{btt}-reduction of norm $2$
from the function $d$ to the function $d\res S$.
(It would be of norm $1$ if not for Case (3), which
required two questions to be asked of the oracle.)

The $1$-reduction $(d\res S) \leq_1 d$ is obvious:
one reduces using the identity function.  To obey the technicalities
of the definition, one should define $d\res S$ to be equal to $d$
on $S$, and equal to $0$ everywhere else, since under a careful
reading, no non-total function can $1$-reduce to a total function.
Fortunately, $S$ is a computable set, so the $1$-reduction can map
each pair $(x,y)\notin S^2$ to a distinct pair $(z,z)$ with $z\notin S$.
\qed\end{pf}
This allows us to prove a corollary about the
distance function on each computable copy of $G$.
\begin{cor}
\label{cor:omegaspec}
There exists a computable connected graph $G_\omega$
whose distance function is (of course) intrinsically approximable from above,
and such that for every function $f$ approximable from above,
there exists a computable graph $G_f$ isomorphic to $G$
whose distance function $d_f$ satisfies $d_f\leq_\twobtt f$
and $f\leq_\onebtt d_f$.
\end{cor}
\begin{pf}
Let $\Gamma$ be a computable function which enumerates
all strictly decreasing sequences in $\omega^{<\omega}$.
Moreover, arrange $\Gamma$ so that every sequence in the range of $\Gamma$ is
equal to $\Gamma(n)$ for infinitely many $n$.
Let $G_\omega$ be a computable presentation
of the standard graph of type $\Gamma$, and $d_\omega$
its distance function.  

Now  let $f:\omega\to\omega$ be any total function
which is approximable from above.  
Write $f=\lim_s g(x,s)$, where $g$ is computable and nonincreasing
in $s$.  For each $x$, let $\Delta_s(2x)=\la g(x,s_0),g(x,s_1),\ldots,g(x,s_k)\ra$,
where $s_0=0$ and $s_{i+1}$ is the least $t\leq s$ (if any)
such that $g(x,t) < g(x,s_i)$.  When there is no such $t$, we set $k=i$,
ending the sequence.  Thus
$$ g(x,s_0)=g(x,s_1-1) > g(x,s_1)= g(x,s_2-1)>g(x,s_2) =\cdots > g(x,s_k)=g(x,s),$$
and the $s_i$ are the stages $\leq s$ at which $g(x,s)$ decreases.
Meanwhile, let $\Delta_s(2x+1)=\Gamma(x)$ for every $x$ and $s$.
The uniformly computable family $\Delta_s$ then has limit $\Delta$ with
$f\leq_\onebtt \Delta$,
since $\Delta(2x)$ is a decreasing string whose final value is $f(x)$,
while $\Delta(2x+1)=\Gamma(x)$ is computable.
Therefore, there is a computable presentation $G_f$ of the standard
graph of type $\Delta$.  We claim that $G_f\cong G_\omega$, and that
the distance function $d_f$ of $G_f$ has $d_f\leq_\twobtt F \leq_\onebtt f$
(hence $d_f\leq_\twobtt f$) and $f\leq_\onebtt d_f$.
This follows from Lemma
\ref{lemma:standarddistance}, since $d_f(a_{2x},b_{2x})=f(x)$
and $d_f(a_{2x+1},b_{2x+1})$ is computable in $x$.
The isomorphism between $G$ and $G_\omega$ is clear, since the range
of $\Delta$ and the range of $\Gamma$ each contains every strictly decreasing
string in $\omega^{<\omega}$ infinitely many times:
every such string equals both $\Gamma(n)$
and $\Delta(2n+1)$ for infinitely many $n$.  This determines
the isomorphism types of the standard graphs $G$ and $G_\omega$
of these types, and they are the same.
\qed\end{pf}

\section{$n$-Approximable Distance Functions}
\label{sec:napprox}

We now turn to graphs in which the distance function
is $n$-approximable from above, for some fixed $n$.
Our goal here is to repeat Corollary \ref{cor:omegaspec}
for the property of $n$-approximability from above,
rather than for arbitrary approximability from above.
Of course, a function which is $n$-approximable from
above must be approximable from above, and therefore
is included in statements such as Theorem \ref{thm:singledegree}.
However, we would like to make the $n$-approximability
intrinsic to the graph, in the following sense
(which is standard in computable model theory).

\begin{defn}
\label{defn:intrinsic}
The distance function on a computable graph $G$
is \emph{intrinsically $n$-approximable from above}
if, for every computable graph $H$ isomorphic to $G$,
the distance function on $H$ is $n$-approximable from above.
$G$ is \emph{relatively intrinsically $n$-approximable from above}
if for every graph $H\cong G$ with domain $\omega$,
the Turing degree of the edge relation on $H$ computes
an $n$-approximation from above to the distance function on $H$.
\end{defn}
One could give the same definition with $\omega$ in place of $n$,
but we have already remarked that the distance function of
every computable graph is approximable from above
(which is to say, $\omega$-approximable from above),
so this would be trivial.  In a graph satisfying
Definition \ref{defn:intrinsic}, there is some structural reason for which
the distance function is always $n$-approximable from above,
no matter how one presents the graph.
(Relative intrinsic $n$-approximability from above says that
this holds even when we consider noncomputable presentations.)
The construction given in Corollary \ref{cor:omegaspec} does
not have this property:  even if we have $n$-approximations to
$d(a_x,b_x)$ and $d(a_y,b_y)$ from above,
we do not get an $n$-approximation to $d(b_x,b_y)$:
this distance will equal the sum $d(b_x,a_x)+2+d(a_y,b_y)$,
and since either summand could decrease as many as $n$ times,
$d(b_x,b_y)$ could decrease as many as $2n$ times.
(This has to do with $d$ being \twobtt \ reducible to the
original function, rather than \onebtt \ reducible.)
So we must revise the format of our graphs.
The next definition does so in several ways,
mainly by adding a second center node and
elongating the paths between the center nodes
and the $a$- and $b$-nodes of each spoke.
Both of these changes will turn out to be essential,
and will result in useful theorems, but as we shall see,
the best result one would hope for remains unproven.

\begin{defn}
\label{defn:elongatedspokes}
Let $\sigma\in\omega^{<\omega}$ be any strictly decreasing nonempty finite string.
An \emph{elongated spoke} of 
type $\sigma$ in a graph consists of the following.
\begin{itemize}
\item
A chain of $(3+\sigma(0))$ nodes (hence of length $2+\sigma(0)$),
beginning with the \emph{top center node}
$u$ and ending with a node called $a$; and
\item
another chain of $(3+\sigma(0))$ nodes, beginning with the \emph{bottom center node}
$v$ (usually visualized sitting below $u$) and ending with
a node called $b$; and
\item
for each $n<|\sigma|$, a path from $a$ to $b$ consisting of $(1+\sigma(n))$ nodes
in addition to $a$ and $b$ (hence of length $2+\sigma(n)$); and
\item
two more paths from $a$ to $b$, of length $2+\sigma(0)$,
in addition to the one already built for $n=0$ in the preceding instruction.
\end{itemize}
These paths do not intersect each other, except at $a$ and $b$, and if $x$ and $y$
are nodes from two distinct paths, then $x$ and $y$ are not adjacent to each other,
nor to $u$, nor to $v$.  The center nodes $u$ and $v$ will belong
to every spoke in the graph, but all other nodes belong only to this spoke,
and will not be adjacent to any node in any other spoke.

For any function $\Gamma$ mapping each $m\in\omega$
to a finite decreasing nonempty string $\sigma_m\in\omega^{<\omega}$,
the \emph{elongated standard graph of type $\Gamma$
} consists of center nodes $u$ and $v$ and, for each $m$,
a spoke of type $\sigma_m$.  (If $\Gamma$ is not injective, then there are just as many
spokes of type $\sigma$ as there are elements in $\Gamma^{-1}(\sigma)$.)
\end{defn}


Here is a picture of three spokes of such a graph, with $\sigma_0=\la 4\ra$,
$\sigma_1=\la 3,1\ra$, and $\sigma_2=\la 2,1,0\ra$.  Notice that the elongation paths
emerging from $u$ and from $v$ have different lengths, determined
by the length of the longest path from $a_n$ to $b_n$.


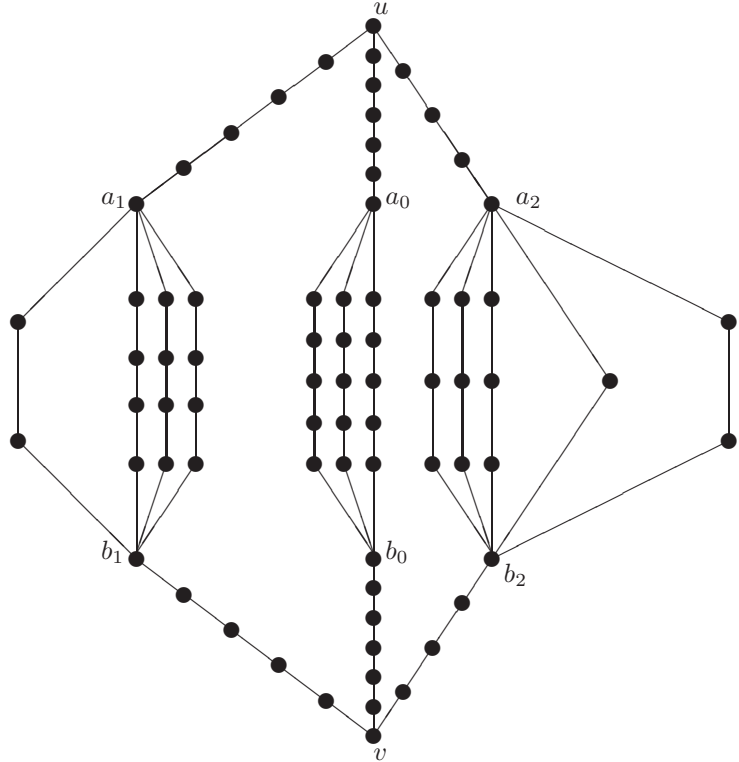
\setlength{\unitlength}{0.62in}
\begin{picture}(7.5,6.6)(-0.1,-1.3)

\put(4,5){\circle*{0.125}}
\put(4,-1){\circle*{0.125}}
\put(4,5.1){$u$}
\put(4,-1.2){$v$}

\put(4,3.5){\circle*{0.125}}
\put(4.1,3.5){$a_0$}
\put(4,2.7){\circle*{0.125}}
\put(4,2.35){\circle*{0.125}}
\put(4,2.0){\circle*{0.125}}
\put(4,1.65){\circle*{0.125}}
\put(4,1.3){\circle*{0.125}}
\put(4,0.5){\circle*{0.125}}
\put(4.1,0.5){$b_0$}

\put(4,0.5){\line(0,-1){1.5}}
\put(4,0.25){\circle*{0.125}}
\put(4,0){\circle*{0.125}}
\put(4,-0.25){\circle*{0.125}}
\put(4,-0.5){\circle*{0.125}}
\put(4,-0.75){\circle*{0.125}}
\put(4,-0.4){\line(0,-1){0.6}}

\put(4,5){\line(0,-1){1.5}}
\put(4,4.75){\circle*{0.125}}
\put(4,4.5){\circle*{0.125}}
\put(4,4.25){\circle*{0.125}}
\put(4,4.0){\circle*{0.125}}
\put(4,3.75){\circle*{0.125}}
\put(4,4.1){\line(0,-1){0.6}}
\put(4,3.5){\line(0,-1){3.5}}

\put(4,3.5){\line(-2,-3){0.5}}
\put(3.5,2.7){\circle*{0.125}}
\put(3.5,2.35){\circle*{0.125}}
\put(3.5,2.0){\circle*{0.125}}
\put(3.5,1.65){\circle*{0.125}}
\put(3.5,1.3){\circle*{0.125}}
\put(3.5,2.7){\line(0,-1){1.4}}
\put(3.5,1.3){\line(2,-3){0.5}}

\put(4,3.5){\line(-1,-3){0.25}}
\put(3.75,2.7){\circle*{0.125}}
\put(3.75,2.35){\circle*{0.125}}
\put(3.75,2.0){\circle*{0.125}}
\put(3.75,1.65){\circle*{0.125}}
\put(3.75,1.3){\circle*{0.125}}
\put(3.75,2.7){\line(0,-1){1.4}}
\put(3.75,1.3){\line(1,-3){0.25}}

\put(4,5){\line(-4,-3){2}}
\put(3.6,4.7){\circle*{0.125}}
\put(3.2,4.4){\circle*{0.125}}
\put(2.8,4.1){\circle*{0.125}}
\put(2.4,3.8){\circle*{0.125}}
\put(2.8,4.1){\line(-4,-3){0.8}}
\put(2,3.5){\circle*{0.125}}
\put(1.7,3.5){$a_1$}
\put(2,2.7){\circle*{0.125}}
\put(2,2.2){\circle*{0.125}}
\put(2,1.8){\circle*{0.125}}
\put(2,1.3){\circle*{0.125}}
\put(2,0.5){\circle*{0.125}}
\put(1.7,0.5){$b_1$}
\put(2,3.5){\line(0,-1){3.0}}
\put(2,0.5){\line(4,-3){2}}
\put(2.4,0.2){\circle*{0.125}}
\put(2.8,-0.1){\circle*{0.125}}
\put(3.2,-0.4){\circle*{0.125}}
\put(3.6,-0.7){\circle*{0.125}}

\put(2,3.5){\line(2,-3){0.5}}
\put(2.5,2.7){\circle*{0.125}}
\put(2.5,2.2){\circle*{0.125}}
\put(2.5,1.8){\circle*{0.125}}
\put(2.5,1.3){\circle*{0.125}}
\put(2.5,2.7){\line(0,-1){1.4}}
\put(2.5,1.3){\line(-2,-3){0.5}}

\put(2,3.5){\line(1,-3){0.25}}
\put(2.25,2.7){\circle*{0.125}}
\put(2.25,2.2){\circle*{0.125}}
\put(2.25,1.8){\circle*{0.125}}
\put(2.25,1.3){\circle*{0.125}}
\put(2.25,2.7){\line(0,-1){1.4}}
\put(2.25,1.3){\line(-1,-3){0.25}}

\put(1,2.5){\circle*{0.125}}
\put(1,1.5){\circle*{0.125}}
\put(2,3.5){\line(-1,-1){1}}
\put(1,2.5){\line(0,-1){1.0}}
\put(1,1.5){\line(1,-1){1}}

\put(4,5){\line(2,-3){1}}
\put(4.25,4.625){\circle*{0.125}}
\put(4.5,4.25){\circle*{0.125}}
\put(4.75,3.875){\circle*{0.125}}
\put(4.6,4.1){\line(2,-3){0.4}}
\put(5,3.5){\circle*{0.125}}
\put(5.2,3.5){$a_2$}
\put(5,2.7){\circle*{0.125}}
\put(5,2.0){\circle*{0.125}}
\put(5,1.3){\circle*{0.125}}
\put(5,0.5){\circle*{0.125}}
\put(5.1,0.3){$b_2$}
\put(5,3.5){\line(0,-1){3.0}}
\put(5,0.5){\line(-2,-3){1}}

\put(4.75,0.125){\circle*{0.125}}
\put(4.5,-0.25){\circle*{0.125}}
\put(4.25,-0.625){\circle*{0.125}}

\put(5,3.5){\line(-2,-3){0.5}}
\put(4.5,2.7){\circle*{0.125}}
\put(4.5,2.0){\circle*{0.125}}
\put(4.5,1.3){\circle*{0.125}}
\put(4.5,2.7){\line(0,-1){1.4}}
\put(4.5,1.3){\line(2,-3){0.5}}

\put(5,3.5){\line(-1,-3){0.25}}
\put(4.75,2.7){\circle*{0.125}}
\put(4.75,2.0){\circle*{0.125}}
\put(4.75,1.3){\circle*{0.125}}
\put(4.75,2.7){\line(0,-1){1.4}}
\put(4.75,1.3){\line(1,-3){0.25}}

\put(7,2.5){\circle*{0.125}}
\put(7,1.5){\circle*{0.125}}
\put(5,3.5){\line(2,-1){2}}
\put(7,2.5){\line(0,-1){1.0}}
\put(7,1.5){\line(-2,-1){2}}

\put(6,2.0){\circle*{0.125}}
\put(5,3.5){\line(2,-3){1}}
\put(6,2.0){\line(-2,-3){1}}

\end{picture}

If $\Gamma$ is computable, then of course we have a computable presentation of
the elongated standard graph of type $\Gamma$.  
As before, though, we will usually be interested in the situation where
$\Gamma=\lim_s \Gamma_s$ and $\Gamma_s$ is computable uniformly in $s$.

\begin{lemma}
\label{lemma:elongatedstandarddistance}
Let $G$ be a computable copy of the elongated standard graph
of some type $\Gamma$, 
with distance function $d$, and let $u$ and $v$
be the centers of $G$.  Then there are computable
functions $p_a$, $p_b$ with domain $G-\{ u,v\}$
which output 
the unique nodes $p_a(x)$ and $p_b(x)$
of valence $>2$ such that $x$, $p_a(x)$,
and $p_b(x)$ are all on the same spoke of $G$
with $d(p_a(x),u) < d(p_b(x),u)$.  We can also compute
whether $x$ lies between $p_a(x)$ and $p_b(x)$, or above
$p_a(x)$, or below $p_b(x)$.

Moreover, the distance function $d$
satisfies $d\leq_\twobtt (d\res S)$ and $(d\res S)\leq_1 d$,
where $S$ is the following set:
$$ S=\Set{\la p_a(x),p_b(x)\ra\in G^2}{x\notin\{ u,v\}}.$$
\end{lemma}
\begin{pf}
Given any $x\in G$ with $x\notin\{ u,v\}$,
we search for a path in $G$ which goes from $u$ through $x$
to $v$ without containing any node more than once.
Then we enumerate $G$ until two nodes on
this path (distinct from $u$ and $v$) are each adjacent to
at least three other nodes.  By the structure of $G$,
these two nodes are the desired $p_a(x)$ and $p_b(x)$,
with $p_a(x)$ being the one closer to $u$ along the path we found,
and from this path we can also determine whether $x$
lies above $p_a(x)$, below $p_b(x)$, or between the two of them.

To compare $d$ with $d\res S$, let $x,y\in G-\{ u,v\}$.
Find the
end points $a_x=p_a(x)$, $a_y=p_a(y)$, $b_x=p_b(x)$,
and $b_y=p_b(y)$ of the spokes containing $x$ and $y$.
We consider first the case where $x$ lies between $a_x$ and $b_x$
and $y$ lies between $a_y$ and $b_y$; the subsequent cases
will then be easier.  These $x$ and $y$ lie on the same spoke in $G$
iff $a_x=a_y$, and if so,
we can then check whether $x$ and $y$
lie on the same path between $p_a(x)$ and $p_b(x)$.
We use $d'(x,a_x)$
to denote the distance from $x$ to $a_x$ along the path
from $a_x$ to $a_y$ through $x$, which we can compute;
similarly for $d'(y,b_y)$, etc.  It is important to note that
$d'(x,a_x)$ may fail to equal $d(x,a_x)$, since there could
be a separate path from $a_x$ to $b_x$ so short
that $d(a_x, b_x)+d'(b_x,x) < d'(a_x,x)$.
Finally, we can readily find the elongation path lengths
of these spokes: the length $l_x$ of the direct
path from $a_x$ to $u$ (and of the path from
$b_x$ to $v$), and the similar length $l_y$ for $y$.


From the structure of $G$, we see that the shortest path
from $x$ to $y$ must be one of the following thirteen paths.
(This requires a combinatorial argument, based on the valences of the nodes
-- which determine the number of options a path has at each point,
given that the path should not go through the same node twice.
Certain other paths are combinatorially possible but are not
on this list; they are discussed below as $P_{13},\ldots,P_{20}$.)
$$\begin{array}{cll}
\text{Path}&\text{Route}&\text{Length}\\
\hline
P_0 & x\text{~to~}a_x\text{~to~}u\text{~to~}a_y\text{~to~}y
& d'(x,a_x)+l_x+l_y+d'(a_y,y)\\
P_1 & x\text{~to~}a_x\text{~to~}u\text{~to~}a_y\text{~to~}b_y\text{~to~}y
& d'(x,a_x)+l_x+l_y+d(a_y,b_y)+d'(b_y,y)\\
P_2 & x\text{~to~}a_x\text{~to~}u\text{~to~}v\text{~to~}b_y\text{~to~}y
& d'(x,a_x)+l_x+d(u,v)+l_y+d'(b_y,y)\\
P_3 & x\text{~to~}a_x\text{~to~}b_x\text{~to~}v\text{~to~}b_y\text{~to~}y
& d'(x,a_x)+d(a_x,b_x)+l_x+l_y+d'(b_y,y)\\
P_4 & x\text{~to~}b_x\text{~to~}v\text{~to~}b_y\text{~to~}y
& d'(x,b_x)+l_x+l_y+d'(b_y,y)\\
P_5 & x\text{~to~}b_x\text{~to~}v\text{~to~}b_y\text{~to~}a_y\text{~to~}y
& d'(x,b_x)+l_x+l_y+d(b_y,a_y)+d'(a_y,y)\\
P_6 & x\text{~to~}b_x\text{~to~}v\text{~to~}u\text{~to~}a_y\text{~to~}y
& d'(x,b_x)+l_x+d(v,u)+l_y+d'(a_y,y)\\
P_7 & x\text{~to~}b_x\text{~to~}a_x\text{~to~}u\text{~to~}a_y\text{~to~}y
& d'(x,b_x)+d(b_x,a_x)+l_x+l_y+d'(a_y,y)\\
\hline
P_8 & x\text{~to~}a_x=a_y\text{~to~}y & d'(x,a_x)+d'(a_y,y)\\
P_9 & x\text{~to~}b_x=b_y\text{~to~}y & d'(x,b_x)+d'(b_y,y)\\
P_{10} & x\text{~to~}a_x=a_y\text{~to~}b_y\text{~to~}y
& d'(x,a_x)+d(a_y,b_y)+d'(b_y,y)\\
P_{11} & x\text{~to~}b_x=b_y\text{~to~}a_y\text{~to~}y
& d'(x,b_x)+d(b_y,a_y)+d'(a_y,y)\\
P_{12} & x\text{~to $y$} & d'(x,y)\\
\end{array}$$
Here paths $P_8$ through $P_{12}$ are separated because they
apply only if $a_x=a_y$
(that is, if $x$ and $y$ lie on the same spoke).  $P_{12}$
only applies if $x$ and $y$ lie on the same path through that spoke.
Paths $P_4$ through $P_7$ may be seen as vertical reflections
of paths $P_0$ through $P_3$.  Every one of these thirteen paths can,
under certain circumstances, be the shortest path from $x$ to $y$.
There are eight other paths which one can define from $x$ to $y$
without repeating any nodes (given that one always takes the shortest route
between $a_x$ and $b_x$, or between $a_y$ and $b_y$, or between
$u$ and $v$, whenever the path route says to go from one
of these nodes to the other).  Here are four of them; the other four
are their vertical reflections.
$$\begin{array}{cl}
\text{Path}&\text{Route}\\
\hline
P_{13} &
x\text{~to~}a_x\text{~to~}u\text{~to~}v\text{~to~}b_y\text{~to~}a_y\text{~to~}y\\
P_{14} &
x\text{~to~}a_x\text{~to~}b_x\text{~to~}v\text{~to~}b_y\text{~to~}a_y\text{~to~}y\\
P_{15} &
x\text{~to~}a_x\text{~to~}b_x\text{~to~}v\text{~to~}u\text{~to~}a_y\text{~to~}y\\
P_{16} &
x\text{~to~}a_x\text{~to~}b_x\text{~to~}v\text{~to~}u
\text{~to~}a_y\text{~to~}b_y\text{~to~}y\\
\end{array}$$
None of these eight paths can be the shortest path from $x$ to $y$.
For example, $P_{13}$ could be shortened by going from $u$ directly
to $a_y$, thereby reducing a subpath of length $d(u,v)+l_y+d(b_y,a_y)$
to a subpath of length $l_y$.  $P_{14}$ and $P_{15}$ also take longer routes than
necessary from $a_x$ to $a_y$, and $P_{16}$ takes a longer route
than necessary from $v$ to $b_y$; similarly for the four vertical reflections.
So, by brute force, we have seen that the shortest path must be one
of $P_0,\ldots,P_{12}$.

Now, given that $d'$ is computable and that $d(u,v)$ is a finite piece of information,
we see that the length of each $P_i$ with $i<13$ is computable from an oracle
for $d\res S$, uniformly in $x$ and $y$, and in particular from just two questions
to that oracle:  the values $d(a_x,b_x)$ and $d(a_y,b_y)$.
So this oracle also allows us to compute $d(x,y)$,
by taking the minimum of those thirteen lengths.
For these $x$ and $y$, therefore, we have a \btt-reduction
of norm $2$.

If $x$ lies between $a_x$ and $a_y$, as above,
but $y=u$, then a similar but easier argument applies:
the shortest path from $u$ to $x$ has length either
$(l_x+d'(a_x,x))$, or $(l_x+d(a_x,b_x)+d'(b_x,x))$,
or possibly $(d(u,v)+l_x+d'(b_x,x))$.
When $y$ lies on an elongation path of length $l_y$
beginning at $u$, we compute $d(a_y,x)$ and $d(u,x)$ as above,
and use them to determine $d(y,x)$.  The reader should
be able to produce a similar argument when $y=v$,
and when $y$ lies on an elongation path beginning at $v$.
Finally, in case neither $x$ nor $y$ lies
between $a_x$ and $a_y$ (resp.\ $b_x$ and $b_y$),
the argument is similar, using the two nodes
at the end of the elongation path containing $x$
and the similar two nodes for $y$, and using these
to take the minimum of the four possible ways
of going from $x$ to $y$.
So we have computed $d(x,y)$ for all possible pairs $(x,y)\in G^2$.
Moreover, the only values we ever required from the
$(d\res S)$-oracle were $d(a_x,b_x)$ and $d(a_y,b_y)$,
and we were able to compute in advance the value
of $d(x,y)$ for each possible answer the oracle might give.
Thus $d\leq_\twobtt (d\res S)$.  The reverse reduction,
$(d\res S)\leq_1 d$, is proven just as in Lemma
\ref{lemma:standarddistance}.
\qed\end{pf}

We now repeat Corollary \ref{cor:omegaspec} for
just the functions $n$-approximable from above.
This is where it becomes clear why we used two centers and
elongated spokes in our graphs in Definition
\ref{defn:elongatedspokes}.  If the graph had only one center $u$,
then the distance $d(b_x,u)$ would be
$n$-approximable from above, but the distance
$d(b_x,b_y)$ would in general only be $(2n)$-approximable
from above: it would decrease whenever either
$d(u,b_x)$ and $d(u,b_y)$ decreased,
since it would equal the sum of these two values.
It may not be immediately clear why having a second center
solves this problem:  for example, the distance $d(a_x,b_y)$ now
depends on both $d(a_x,v)$ and $d(b_y,u)$,
each of which could decrease as many as $n$ times.
However, $d(a_x,b_y)$ is now equal to the minimum
of these two values (up to a computable difference),
not to their sum.  With the minimum, we will
eventually avoid this difficulty
by turning to the countdown function for an arbitrary
function which is $n$-approximable from above.
First, though, we prove the basic result.


\begin{thm}
\label{thm:nspec}
For every $n<\omega$, there exists a computable connected graph $G$
whose distance function $d_G$ is intrinsically $2n$-approximable from above and
such that,
for every function $f$ which is $n$-approximable from above,
there is some computable graph $H\cong G$ whose distance function
$d_H$ has $d_H\leq_\twobtt f \leq_\onebtt d_H$.
\end{thm}
\begin{pf}
For $n=0$, just take $G$ to be any computable graph with an
intrinsically computable distance function.
Since all total computable functions are \btt-equivalent
with norm $1$, this suffices.  (For example,
the complete graph on the domain $\omega$ could serve as $G$.)

For $n>0$, define a computable function $\Gamma:\omega\to \omega^{\leq n}$
so that the range is $\Gamma$ contains exactly those tuples
of length $\leq n$ which are strictly decreasing;
moreover, ensure that every such tuple
has infinite preimage in $\omega$ under $\Gamma$.
Let $G$ be a computable presentation of the elongated
standard graph of type $\Gamma$.
We claim that this $G$ instantiates the theorem.

To see that the distance function $d_H$ on an arbitrary
computable graph $H\cong G$ is $(2n)$-approximable from above,
notice that just as in Lemma \ref{lemma:elongatedstandarddistance},
we can compute the functions $p_a$ and $p_b$ for $H$.
Since every $\sigma\in\rg{\Gamma}$ has $|\sigma|\leq n+1$,
the paths from $p_a(x)$ to $p_a(y)$ within the spoke of any $x\in H$
have at most $(n+1)$ distinct lengths, so that $d_H(p_a(x),p_a(y))$
can be approximated from above with at most $n$ mind changes.
Then we apply Lemma \ref{lemma:elongatedstandarddistance} to see
that the values $d_H(p_a(x),p_b(x))$ and $d_H(p_a(y),p_b(y))$
determine $d_H(x,y)$, by taking minimums of the thirteen paths
as described there.  So we simply approximate $d_H(p_a(x),p_b(x))$
and $d_H(p_a(y),p_b(y))$ from above, and whenever either
approximation is reduced, we reduce our approximation of
$d_H(x,y)$ accordingly.  This gives an approximation of $d_H(x,y)$
from above with at most $2n$ changes:  the $\leq n$ stages at which
the approximation to $d_H(p_a(x),p_b(x))$ changed,
and the $\leq n$ stages at which $d_H(p_a(y),p_b(y))$ changed.

Next, given an arbitrary function $f$ which is $n$-approximated
from above by $g$, we build a graph $H$ by starting with $G$
and adjoining, for each $x\in\omega$, a spoke of type
$$\sigma_x=\la g(x,s_0),g(x,s_1),\ldots,g(x,s_k)\ra,$$
where $s_0=0$ and each $s_{i+1}=\min\Set{s}{g(x,s)<g(x,s_i)}$.
It is clear how one can do this effectively, starting with a spoke of type
$\la g(x,s_0)\ra$ and extending it to a larger spoke
each time a new value of $g(x,s)$ appears.
Of course, we have $k\leq n$, and so $\sigma_x=\Gamma(m)$
for infinitely many $m$.  Therefore, the new spokes we
add do not change the isomorphism type, but leave $H\cong G$.

One new problem arises:  it is no longer immediate that
$d(a_m,b_m)=2+f(m)$.  To see the problem,
notice that, whereas in the standard presentation of a graph
(with only one center), it was clear that the shortest path from
$a_m$ to $b_m$ was one of the paths between them within that
spoke.  Now, however, there is a path from $a_m$ to $u$,
then through another spoke to $v$, then up to $b_m$,
and the length of this path is $l_m+d(u,v)+l_m$.
The elongation paths were given their length (call
it $l_m$, for the $m$-th spoke) precisely to ensure
that this alternative is not the shortest path from $a_m$ to $b_m$,
and since $l_m$ was chosen to be the length of the longest
path from $a_m$ to $b_m$ within the spoke, it is clear that
this has been accomplished.  So $f(m)+2$ really does equal
$d(a_m,b_m)$.  It follows that $f\leq_\onebtt d_H$ as required,
thanks to Lemma \ref{lemma:elongatedstandarddistance},
which now gives the reducibilities $d_H\leq_\twobtt f\leq_\onebtt d_H$,
exactly as desired.
\qed\end{pf}

To strengthen Theorem \ref{thm:nspec},
we would like to make $G$ have a distance function which
is intrinsically $n$-approximable from above.
The proof given above does not accomplish this.
In particular, for the opposite end points $a_j$ and $b_k$
of two distinct spokes in $H$, with $j\neq k$, the formula
for $d_H(a_j,b_k)$ involves a minimum of eight different
values, some of which depend on $d_H(a_j,b_j)$ and
others on $d_H(a_k,b_k)$.  Each of these two
distances could decrease as many as $n$ times
as our computable approximations to $f(j)$ and $f(k)$
decrease, and so the minimum could decrease as many as $2n$ times.
This is the same problem we would have had using the simpler
(single-center, non-elongated) graphs of Definition
\ref{defn:spokes}, except that there the problem involved
a sum, not a minimum.  One is led to wonder what purpose
Definition \ref{defn:elongatedspokes} served.
The proof of the following theorem gives the answer.

\begin{thm}
\label{thm:ncountdownspec}
For every $n<\omega$, there exists a computable connected graph $G$
whose distance function $d_G$ is intrinsically $n$-approximable from above and
such that,
for every function $f$ with a computable $n$-approximation $g$ from above,
there is some computable graph $H\cong G$ whose distance function $d_H$ has
$$d_H\leq_\twobtt \lim_s c \leq_\onebtt d_H,$$
where $c$ is the countdown function for $g$ with bound $n$.
\end{thm}

So we have achieved intrinsic $n$-approximability from above for the distance
function, while still allowing the distance function -- under a coarser
reducibility -- to realize all functions $n$-approximable from above.
Since $f\equiv_{\bT} \lim_s c$, the theorem shows that $f\equiv_{bt} d_H$,
where $H$ is the graph built for $f$, but the \btt-reducibility
and the specific norms have been lost.  Before proving the theorem,
we summarize this as a corollary.
\begin{cor}
\label{cor:bTspec}
For every $n$, there exists a computable connected graph $G$
such that the \bT-degree spectrum of the distance function on $G$
contains exactly the \bT-degrees $n$-approximable from above
(that is, those \bT-degrees which contain a function $n$-approximable
from above).
\qed\end{cor}

\begin{pftitle}{Proof of Theorem \ref{thm:ncountdownspec}}
Assume $n>0$, since otherwise the description
in the proof of Theorem \ref{thm:nspec} suffices.
Let $\Gamma$ be a computable function with
$$ \Gamma(0) = \la n\ra,~\Gamma(1)=\la n,n-1\ra,\ldots,~\Gamma(n)=\la n,n-1,\ldots,0\ra$$
and with $\Gamma(m+n+1)=\Gamma(m)$ for all $m\in\omega$,
so that each of these strings appears infinitely
often in the range of $\Gamma$.
%
Our graph $G$ is a computable presentation
of the standard graph of type $\Gamma$.

Our first goal is to show that for every computable graph
$H$ isomorphic to $G$, the distance function
$d(x,y)$ of $G$ is always $n$-approximable from above.
For $x,y\in G$, the proof of Lemma
\ref{lemma:elongatedstandarddistance} gives $d_H(x,y)$ as the minimum
of finitely many values, and that each of those values
being a sum of computable values along with $d_H(p_a(x),p_b(x))$
or $d_H(p_a(y),p_b(y))$.  Crucially, though, none of those
values (whose minimum we take) involves either $d_H(p_a(x),p_b(x))$
or $d_H(p_a(y),p_b(y))$ more than once.  Therefore, whenever
$d_H(p_a(x),p_b(x))$ decreases by $1$, certain of the values decrease
by $1$ and the rest stay unchanged; likewise, whenever
$d_H(p_a(y),p_b(y))$ decreases by $1$, certain other values decrease
by $1$ and the rest stay unchanged.  Moreover, the structure of $G$
shows that our approximations to
$d_H(p_a(x),p_b(x))$ and $d_H(p_a(y),p_b(y))$ never
decrease by more than $1$ at any stage.  (More exactly,
the approximation begins at $n$, so it can decrease at most $n$
times.  If at stage $s$ we thought $d_H(p_a(x),p_b(x))=7$,
and at stage $s+1$ we find a path of length $4$ between
$p_a(x)$ and $p_b(x)$, we can think of this as three
separate decreases by $1$.  In fact, the structure of $G$ is such
that in this case there will indeed be paths of length $5$ and
of length $6$ between $p_a(x)$ and $p_b(x)$, even though the path
of length $4$ appeared first.)  This allows us to
apply the following lemma, whose proof is a straightforward induction on $n$.

\begin{lemma}
\label{lemma:dropby1}
If $g(x,s)$ satisfies $g(x,s)-1\leq g(x,s+1)\leq g(x,s)$
for all $x$ and $s$, and $h(x,s)$ does the same,
and if $g(x,0)=h(x,0)=n$ and $C$ and $D$ are
arbitrary constants, 
then
$$|\Set{s}{\min (g(x,s+1)+C,h(x,s+1)+D)<\min (g(x,s)+C,h(x,s)+D)}|\leq n.$$
(By induction, the same then holds for
a minimum of arbitrarily many functions with these properties.)
\qed\end{lemma}
So $d(x,y)$, being a minimum of exactly this type,
is also $n$-approximable from above.  Our use of countdown
functions enabled us to use the family $\Gamma$ whose
member strings never decrease by more than $1$

Next, consider any function $f$ which has an $n$-approximation
$g$ from above.  Let $c$ be the countdown function for this $g$
and for the constant bound $n$ on changes to $g$.  This was
defined in Theorem \ref{thm:countdown}:  $c(x,0)=n$,
and $c(x,s+1)=c(x,s)-1$ iff $g(x,s+1)<g(x,s)$,
with $c(x,s+1)=c(x,s)$ otherwise.  So $c$ keeps track
of the number of changes $g$ is still permitted to make
as it approximates $f(x)$, and clearly $c$ satisfies
the property mentioned above of never decreasing by more than $1$.
Recall that $f\equiv_{\bT}\lim_s c$, although stronger equivalences,
such as \tt-equivalence, may fail to hold.

Our construction of the graph $H$ for this $f$
mirrors that of Theorem \ref{thm:nspec}, only using the countdown
function $c$ in place of the computable approximation $g$ to $f$ itself.
We simply start with the computable graph $G$
(whose distance function is also computable)
and, for each $n$, add a new elongated spoke of type
$\la n,n-1,\ldots,k\ra$, where $k=\lim_s c(n,s)$.
Since there were already infinitely many spokes of this type
in $G$, the addition of one (or even infinitely many) more
does not change the isomorphism type; thus $H\cong G$.
Moreover, it is clear that we can add this spoke in a computable fashion:
it starts as a spoke of type $\la n\ra$, then has a path added
and becomes a spoke of type $\la n,n-1\ra$ when and if
we find an $s$ with $c(n,s)=n-1$, and so on.
By Lemma \ref{lemma:elongatedstandarddistance},
we have $d_H \leq_\twobtt\lim_s c\leq_\onebtt d_H$..
Unfortunately, $f$ cannot be substituted for $\lim_s c$
in these reductions, because in general we only have
$f\equiv_{\bT} \lim_s c$, and so we conclude,
as claimed by the corollary, that $f\equiv_{\bT} d_H$
for this graph $H$.
\qed\end{pftitle}

The same strategy could have been used in the proof
of Corollary \ref{cor:omegaspec}, of course.  There,
however, it was not necessary:  the distance function
on any computable graph is always $\omega$-approximable from above.
Moreover, if $x$ and $y$ in that graph $G_\omega$
(or a copy of it) lie on distinct spokes, then an upper
bound on the number of changes in the natural approximation
to $d(x,y)$ can be given just by adding the (computable)
upper bounds on the number of changes in the approximations
to $d(a_x,b_x)$ and to $d(a_y,b_y)$.  So it seemed
superfluous to convert the approximable-from-above
function $f$ given there to a function which never decreases by more than $1$,
and indeed, not converting it allowed us to retain stronger
reducibilities between $f$ and $d$.

The next natural question, which remains open,
is the existence of a computable connected graph
whose distance function is intrinsically
$n$-approximable from above, but which,
for every $f$ $n$-approximable from above,
has a computable copy $H$ with distance
function $d_H\leq_\twobtt f\leq_\onebtt d_H$.

Of course, stronger reducibilities between $f$
and $d_H$ would be most welcome as well.
Persistently throughout these results, we have had to allow
norm $2$ for the \btt-reduction from the distance function
to the function being encoded, even when
the reverse reduction could be shown to have \btt-norm $1$.
This seems to be a condition intrinsic to the notion
of the distance function.  When $f \leq_\onebtt d_H$,
and $m$ is fixed, there is a single pair $(x,y)$ of nodes in $H$
which determines $f(m)$.  There must be some separate
pair $(x',y')$ determining some other value $f(m')$
(unless $f$ is computable),
and then, in $H$, one can usually find nodes $w$ and $z$
such that $d(w,z)$ depends on both $d(x,y)$ and $d(x',y')$,
either via a sum (if the shortest path from $w$ to $z$
goes through $x$, then $y$, then $x'$, then $y'$),
or via a minimum (if there is one path from $w$ to $z$
which goes through $x$ and $y$, and a separate path
going through $x'$ and $y'$).  In both these cases,
$d(w,z)$ requires two pieces of information from $f$,
leading to a \btt-reduction of norm $2$ at best.
We would be significantly interested in any way
of developing this analysis into a proof of the following.
\begin{conj}
\label{conjecture:nspec}
If $G$ is a computable connected graph whose distance
function is intrinsically $n$-approximable from above, with $n>0$,
then there exists some function $f$ which is $n$-approximable from above
and such that, for every computable graph $H\cong G$ with distance
function $d_H$, we have $d_H\not\equiv_\onebtt f$.
\end{conj}
It would follow that the \onebtt-degree spectrum
of the distance function cannot contain exactly those
degrees which are $n$-approximable from above.

\section{Directed Graphs}
\label{sec:directed}

One can repeat the questions from this paper
in the context of computable directed graphs,
rather than the symmetric graphs we have used.
In a \emph{directed graph}, each edge
between vertices $x$ and $y$ has a specific orientation:
it points either from $x$ to $y$, or from $y$ to $x$.
(It is allowed for there to exist two edges between $x$
and $y$, one pointing in each direction.)  Of course,
the orientations of the edges must be computable.
In this context, one speaks of a \emph{directed path}
from $x$ to $y$ as a finite sequence of nodes
$x=x_0,x_1,\ldots,x_n=y$ such that for all $i<n$,
there is an edge from $x_i$ to $x_{i+1}$.
The directed graph $G$ is \emph{connected} if, for every
$x,y\in G$, there is a directed path from $x$ to $y$,
and in this case the \emph{directed distance} from $x$
to $y$ is always defined:  it is the length of the shortest
directed path from $x$ to $y$.  Again, this function
is intrinsically approximable from above, and it appears to us
that the constructions in this paper work equally well for directed
graphs (modulo a few considerations, such as adding directed
paths from $b$ to $u$ in the standard case, and from
$v$ to $u$ in the elongated case, so as to make the
graph connected).  However, with directed graphs
we can accomplish more than has already been proven here
for symmetric graphs.  In particular, it is much easier to realize
the goal we set for ourselves in Section \ref{sec:napprox}.

\begin{thm}
\label{thm:directed}
There exists a computable directed graph $G$ whose distance function
is intrinsically $n$-approximable from above, yet such that
every $f$ which is $n$-approximable from above is \onebtt-equivalent
to the distance function on some computable directed graph isomorphic to $G$.
\end{thm}
\begin{pf}
Fix $n$.  The directed graph $G$ will be a version of the
standard (undirected) graph on a collection of spokes.
As usual, we let $\Gamma$ enumerate all strictly decreasing
sequences $\sigma\in\omega^{\leq n+1}$ of length
at most $n+1$, and we assume that this enumeration
repeats every such $\sigma$ infinitely often.

A directed spoke of type $\sigma$ looks somewhat like an ordinary
spoke of type $\sigma$.  The two main differences are that
we use $u$ itself as the top node of the directed spoke, rather than
having a top node $a$ adjacent to $u$, and that we include
a directed path, of length $(3+\sigma(0))$, from the bottom node $b$
back to $u$.  The latter modification is necessary in order for
this directed graph to be connected.  The reason for the former will be
explained after the proof.

In between $u$ and $b$, the directed spoke contains three directed paths
from $u$ to $b$ of length $(2+\sigma(0))$, and, for each $i$
with $0<i<|\sigma|$, another directed path from $u$ to $b$
of length $(2+\sigma(i))$.  Thus the final value of $\sigma$
is the length of the shortest directed path from $u$ to $b$ in this spoke
(and there will be no directed paths from $u$ to $b$ through other spokes).
With this, we have described the directed spoke entirely.
The directed graph $G$ contains one directed spoke of type
$\sigma=\Gamma(m)$, with bottom node $b_m$, for each $m\in\omega$.

Now $u$ is the only node anywhere in $G$ with more than one
edge coming out of it.  So, for any computable $H\cong G$, the
same holds for some node $u_H$.  Therefore, there is little difficulty
in choosing the shortest directed path from an arbitrary node $x\in H$
to another one $y$:  the only choice in finding the path arises
when/if one reaches $u_H$.  Starting at $x$, one follows the unique
directed edge emerging from $x$, then the one emerging from that node,
etc., until one reaches either $y$ or $u_H$.  If one has reached $y$,
then the length of the path so far is clearly the directed distance
from $x$ to $y$ in $H$.  Otherwise, one then determines on which spoke
of $H$ $y$ lies (which is computable), and fixes the bottom node $b_m$ of that spoke.
If $y$ lies on a directed path from $u_H$ to $b_m$, the one follows that path
until reaching $y$, and this is the shortest directed path from $x$ to $y$.
If $y=b_m$ or $y$ lies on the directed path from $b_m$ to $u_H$, then one
follows the shortest path from $u_H$ to $b_m$, and then on to $y$.
Here arises the only ambiguity: choosing the shortest directed path
from $u_H$ to $b_m$.  The directed distance from $x$ to $y$ is the
length of this path, plus the lengths of the (already determined)
paths from $x$ to $u_H$ and from $b_m$ to $y$.  Therefore,
the distance function $d_H$ on $H$
is \onebtt-reducible to the function $f(m)=d(u_H,b_m)-2$.
Conversely, this function $f$ has $f\leq_1 d_H$.  But in any computable
copy $H$ of $G$, this $f$ is $n$-approximable from above,
since one simply finds the three paths of length
$(2+\sigma(0))$ from $u_H$ to $b_m$ (where $\sigma=\Gamma(m)$),
thereby identifying $b_m$, and then waits for shorter directed
paths to appear -- which will happen at most $n$ times,
since $|\sigma|\leq n+1$.

Thus $d$ is intrinsically $n$-approximable from above.
The converse is exactly the same construction we have executed
previously.  Given any $f$ which is $n$-approximable from above,
say via some computable $g(m,s)$, we start with a computable copy of $G$
and extend it as follows to a directed graph $H$.
For each $m$, add to $G$ a directed spoke of type $\sigma_m$,
where $\sigma_m(0)=g(m,0)$ and $\sigma_m(i+1)$ is defined
iff there is an $s$ with $g(m,s)<\sigma_m(i)$, in which case
$\sigma_m(i+1)=g(m,s)$ for the least such $s$.  Since $g$ is an
$n$-approximation, we have $|\sigma_m|\leq n+1$, and the last
value of $\sigma$ is $f(m)$, so $d_H(u,b_m)=f(m)$.  For bottom
nodes $b$ of directed spokes in the original graph $G$ within $H$,
$d_h(u,b)$ is computable, since in $G$ we know the type $\sigma$
of each such directed spoke.  Hence, by the same argument
as in the preceding paragraph, $d_H$ is \onebtt-equivalent
to the function $f$, as desired.
\qed\end{pf}

We note here that the conflation of $u$ with the top nodes
$a$ of the directed spokes was necessary.  Had the $a$'s been part of these
directed spokes, then the computation of the distance from
the $a_k$ on spoke $k$ to the $b_m$ on spoke $m$
would have required knowing both $d(a_k,b_k)$
and $d(a_m,b_m)$, hence would have required a \btt-reduction
of norm $2$.  However, the trick of eliminating the nodes $a$
does not allow us to prove Theorem \ref{thm:directed}
for symmetric graphs, since with no orientation on the edges,
the computation of the distance from one bottom node
to another still requires questions about the distance
from top to bottom on two different spokes.

\section{Related Topics}
\label{sec:questions}

For graphs with infinitely many connected components,
the distance function is $(\omega+1)$-approximable from above,
assuming we allow $\infty$ as the distance between
any two nodes in distinct components.
One approximates the distance function $d(x,y)$
at stage $0$ by $g(x,y,0)=\omega$ (or $\infty$),
which will continue to be the value as long as
$x$ and $y$ are not known to be in the same connected component.
Meanwhile, we search systematically for a path from
$x$ to $y$ within increasing finite subgraphs of $G$,
and if we find one, say of length $l$, at some stage $s$,
then we set $g(x,y,s)=l$, and then continue
exactly as in the connected case, searching for shorter paths.
It is clear that this distance function $d$ is therefore
$(\omega+1)$-approximable from above, in the obvious
definition, provided that one allows $\infty$ as an
output of the function.  (If $\infty$ is not allowed,
then no notion of $(\omega+1)$-approximability
from above makes sense and distinguishes the concept
from $\omega$-approximability from above.)

In a different context, recent work by Steiner in
\cite{S12} has considered the number of realizations
of various algebraic types within a computable structure,
and has asked in which cases one can put a computable
upper bound on the number of realizations of each algebraic
type.  Over the theory \textbf{ACF$_0$}, for example,
an algebraic type is generated by
the formula $p(X)=0$, where $p$ is a polynomial
irreducible over the ground field, and the degree
of the polynomial is an upper bound for the number of realizations
of this type in an arbitrary field (not necessarily
algebraically closed).  Since the minimal
polynomial of the element over the prime subfield can be found
effectively, one can compute such an upper bound,
whereas for other computable algebraic structures considered
by Steiner, no such upper bound exists.
The function counting the number of realizations of each algebraic type
in a computable algebraic structure is approximable from below,
and when a computable upper bound exists, this function
becomes the dual of a function approximable from above,
exactly as described in Definition \ref{defn:dual}.
Therefore, the theorems proven in Section \ref{sec:AFA}
apply to such functions.  On the other hand,
when there is no computable upper bound, the standard
results about functions approximable from below apply,
and we saw in Section \ref{sec:AFA} that these results
differ in several ways from the results when a computable
bound does exist.

To close, we ask what connection, if any, there might be between
distance functions on computable graphs and Kolmogorov complexity.
Is it possible that Kolmogorov complexity can be presented
as the distance function on some computable graph?
(This is a different matter than using Kolmogorov complexity to
construct structures with prescribed model-theoretic properties, as in
\cite{KSS}.)
It might be useful to fix one node $e\in G$ -- call it
the \emph{Erd\"os node} -- and to consider $d(e,x)$,
a unary function on $G$, in place of the full distance function;
in this case one could directly build a computable graph $G$
and a computable function $f:\omega\to G$ such that
$d(e,f(n))$ is exactly the Kolmogorov complexity of $n$.
Having done so, one could then ask about other
computable copies of $G$:  does the distance function
on those copies correspond to Kolmogorov complexity
under some different universal (prefix-free?) machine?
Right now, this question is not well-formed, and
there is no obvious reason to expect to find any connections
at all between these topics, except for their common use
of functions approximable from above, and their common
triangle inequalities.  (If one knows the Kolmogorov complexity
of binary strings $\sigma$ and $\tau$, one gets an
upper bound on the Kolmogorov complexity of the concatenation
$\sigma\hat{~}\tau$.)  However, any connection
that might arise would be a potentially fascinating link
between algorithmic complexity and computable model theory.


\end{document}